\documentclass[12pt]{amsart}

\usepackage{euscript, graphicx, epstopdf}
\usepackage{cases}
\usepackage{mathrsfs}
\usepackage{bbm}
\usepackage{amssymb}
\usepackage{txfonts}
\usepackage{amscd}
\usepackage{amsfonts,latexsym,amsmath,amsxtra,mathdots,latexsym,mathabx}
\usepackage[all,cmtip]{xy}
\usepackage{color}
\usepackage{colordvi}
\usepackage{multicol}
\usepackage{hyperref}
\usepackage{tikz}
\usepackage{float}
\usepackage{setspace, floatflt}
\usepackage{eufrak}
\usepackage{hyperref}

\allowdisplaybreaks

\newcommand{\BA}{{\mathbb {A}}}

\newcommand{\BC}{{\mathbb {C}}}
\newcommand{\BD}{{\mathbb {D}}}

\newcommand{\BG}{{\mathbb {G}}}

\newcommand{\BK}{{\mathbb {K}}}

\newcommand{\CF}{{\mathcal {F}}}

\newcommand{\CO}{{\mathcal {O}}}

\newcommand{\FM}{{\mathfrak {M}}}

\newcommand{\FZ}{{\mathfrak {Z}}}

\newcommand{\Ff}{{\mathfrak {f}}}
\newcommand{\Fg}{{\mathfrak {g}}}

\newcommand{\Fz}{{\mathfrak {z}}}

\newcommand{\RM}{{\mathrm {M}}}

\newcommand{\D}{{\mathcal {D}}}

\newcommand{\cn}{{\mathrm {CN}}}
\newcommand{\GL}{{\mathrm {GL}}}

 \newcommand{\Tr}{{\mathrm{Tr}}}

\newcommand{\Lie}{{\mathrm {Lie}}}
\newcommand{\M}{{\mathrm {M}}}
\newcommand{\wf}{{\mathrm {WF}}}
\newcommand{\supp}{{\mathrm {supp}}}

\def\Hom{\mathrm{Hom}}
\def\dim{\mathrm{dim}}

\def\Stab{\mathrm{Stab}}
\def\Mat{\mathrm{Mat}_2}

\def\gl{\mathfrak{gl}}

\def\-{^{-1}}

\makeatletter
\g@addto@macro\normalsize{\setlength\abovedisplayskip{3pt}}
\makeatother

\makeatletter
\g@addto@macro\normalsize{\setlength\belowdisplayskip{3pt}}
\makeatother

\newcommand{\delete}[1]{}

\theoremstyle{plain}

\newtheorem{thm}{Theorem}[section] \newtheorem{cor}[thm]{Corollary}

\newtheorem{lem}[thm]{Lemma}  \newtheorem{prop}[thm]{Proposition}

 \newtheorem{defn}[thm]{Definition}

\newtheorem {rem}[thm]{Remark}

\numberwithin{equation}{section}

\renewcommand{\qedsymbol}{Q.E.D.}

\begin{document}
	
\title[Uniqueness of the Ginzburg-Rallis Model]{Uniqueness of the Ginzburg-Rallis Model: the $p$-adic Case}
	
\date{\today}
	
\author{Dihua Jiang}	
\address{School of Mathematics, University of Minnesota, 206 Church St. S.E., Minneapolis, MN 55455, USA.}
\email{jiang034@umn.edu}

\author{Zhaolin Li}	
\address{School of Mathematics, University of Minnesota, 206 Church St. S.E., Minneapolis, MN 55455, USA.}
\email{li001870@umn.edu}

\author{Guodong Xi}	
\address{School of Mathematics, University of Minnesota, 206 Church St. S.E., Minneapolis, MN 55455, USA.}
\email{xi000023@umn.edu}


\date{\today}


\thanks{The research of this paper is supported in part by the NSF Grant DMS-2200890.}

	
\begin{abstract}
We prove the uniqueness of the Ginzburg-Rallis models over $p$-adic local fields of characteristic zero, which completes the local uniqueness problem for the Ginzburg-Rallis models starting from the work of C.-F. Nien in \cite{MR2709083} that proves the non-split case, and the work of D. Jiang, B. Sun and C. Zhu in \cite{MR2763736} that proves the general case over Archimedean local fields. Our proof extends the strategy of \cite{MR2763736} to the $p$-adic case with the help of the refined structure of the wavefront sets of $\Fz$-finite distributions as developed by A. Aizenbud, D. Gourevitch and E. Sayag in \cite{MR3406530}.
\end{abstract}

	\maketitle
	\tableofcontents

\section{Introduction}\label{introduction}

The Ginzburg-Rallis model for $\GL_6$ was introduced by D. Ginzburg and S. Rallis in their paper (\cite{MR1795291}) to characterize conjecturally the non-vanishing of the central value of the exterior cube  
$L$-function $L(s,\pi,\wedge^3)$ associated with an irreducible cuspidal automorphic representation $\pi$ of $\GL_6(\BA)$, where $\BA$ is the ring of adeles of a number field $k$. Since the exterior 
cube representation $\wedge^3$ of the dual group $\GL_6(\BC)$ is of symplectic type, from the general framework of Y. Sakellaridis and A. Venkatesh (\cite{MR3764130}), 
the central value $L(\frac{1}{2},\pi,\wedge^3)$ should be expressed by a square of an automorphic Ginzburg-Rallis period, up to a meaningful constant. We refer to a recent work of A. Abdurrahman and A. Venkatesh for more detailed discussions when $k$ is a function field (\cite{abdurrahman2023symplectic}). 

The local theory of the Ginzburg-Rallis models was first discussed in \cite{MR2402684}, following the philosophy of the local Gross-Prasad conjecture for special orthogonal groups 
(\cite{MR1186476, MR1295124}). 

Let $F$ be a local field of characteristic zero. Let $G:=\GL_6$ be the general linear group consisting of $6\times 6$-invertible matrices.  Let $B=TU$ be the fixed Borel subgroup of $G$, consisting of all upper 
triangular matrices, with $T$ the maximal $F$-split torus that consists of all the diagonal matrices and with $U$ the unipotent radical of $B$. Let $P_{2^3}=M_{2^3}U_{2^3}$ be the standard parabolic subgroup of $G$ 
associated with the partition $[2,2,2]$, whose elements are of the form 
\begin{align}\label{P222}
\begin{pmatrix}
        a_1 & b & d \\ 0 & a_2 & c\\ 0 & 0 & a_3
\end{pmatrix}\in\GL_6
\end{align}
with $a_1,a_2,a_3\in\GL_2$ and $b,c,d\in \Mat$, the space of all $2\times2$-matrices. The Levi subgroup $M_{2^3}$ of $P_{2^3}$ is isomorphic to the product of three $\GL_2$'s and the unipotent 
radical $U_{2^3}$ consists the following matrices:
\[
\begin{pmatrix}
        1_2 & b & d \\ 0 & 1_2 & c\\ 0 & 0 & 1_2
\end{pmatrix}\in P_{2^3}.
\]
It is clear that 
\[
U_{2^3}/[U_{2^3},U_{2^3}] \cong \Mat\oplus\Mat.
\]

Let $\psi_F$ be any fixed non-trivial additive character of $F$. It follows that characters $\psi$ of $U_{2^3}(F)$ are parameterized by $(A,B)\in\Mat(F)\oplus\Mat(F)$ by 
\[
\psi(u)=\psi_F(\Tr(Ab+Bc))
\]
for any $u=u(b,c,d)\in U_{2^3}(F)$. Under the adjoint action of $M_{2^3}(F)$ on $U_{2^3}(F)$, for any $m=m(a_1,a_2,a_3)\in M_{2^3}(F)$, the character $Ad^*(m)(\psi)$ is parameterized by 
\[
(a_1^{-1}Aa_2,a_2^{-1}Ba_3)\in \Mat(F)\oplus\Mat(F).
\]
Hence the $M_{2^3}(F)$-open orbit of the characters of $U_{2^3}(F)$ has a (unique) representative $\psi_{U_{2^3}}$ which is parameterized by $A=B=I_2$, i.e. 
\begin{align}\label{psiU}
    \psi_{U_{2^3}}(u(b,c,d))=\psi_F(\Tr(b+c)). 
\end{align}
It is clear that the stabilizer $\mathrm{Stab}_{M_{2^3}}(\psi_{U_{2^3}})$ is isomorphic to $\GL_2$, which embeds diagonally into $M_{2^3}$. The Ginzburg-Rallis subgroup of $\GL_6$ is defined to be 
\begin{align}\label{GRgroup}
    S:=U_{2^3}\rtimes\Stab_{M_{2^3}}(\psi_{U_{2^3}}),
\end{align}
whose elements can be written as 
\[
s=s(a,b,c,d)=\begin{pmatrix} 1_2 & b & d \\0 & 1_2 & c \\ 0 & 0 & 1_2 \end{pmatrix}\begin{pmatrix} a & 0& 0\\0&a & 0 \\ 0 & 0 & a \end{pmatrix} 
\in P_{2^3}.
\]
For any (unitary) character $\chi$ of $F^\times$, one may define, with the fixed character $\psi_F$ of $F$, a character $\chi_S$ of the Ginzburg-Rallis subgroup $S(F)$ by 
\begin{align}\label{GRchi}
    \chi_S(s(a,b,c,d)):=\psi_F(\Tr(b+c))\cdot\chi(\det a).
\end{align}
We may call $(G,S,\chi_S)$ the $F$-split Ginzburg-Rallis triple for the Ginzburg-Rallis model of $G=\GL_6$. As introduced in \cite{MR1795291}, one may define the Ginzburg-Rallis triple 
$(\GL_3(\BD),S_\BD,\chi_{S_\BD})$ for each quaternion algebra $\BD$ over $F$. It is clear that when $\BD$ is $F$-split, one gets the $F$-split Ginzburg-Rallis triple 
$(\GL_6,S,\chi_S)$ as defined above. 

Take $G=\GL_3(\BD)$ for any quaternion algebra $\BD$ over $F$. Define $\Pi_F(G)$ to be the set of equivalence classes of irreducible admissible representations $\pi$ of $G(F)$. When $F$ is Archimedean, we 
assume that $\pi\in\Pi_F(G)$ are of Casselman-Wallach type (see \cite{MR1170566}). One of the basic problems in the local theory of the Ginzburg-Rallis models is to establish the 
following multiplicity one theorem. 

\begin{thm}[Local Uniqueness]\label{thm:main}
Let $G=\GL_3(\BD)$ for any quaternion algebra $\BD$ over a local field $F$ of characteristic zero. For any $\pi\in\Pi_F(G)$, the multiplicity of the Ginzburg-Rallis model associated with the Ginzburg-Rallis triple $(\GL_3(\BD),S_\BD,\chi_{S_\BD})$
\[
m(\pi,\chi_{S_{\BD}}):=\dim\;\Hom_{S_{\BD}(F)}(\pi,\chi_{S_{\BD}})
\]
is at most one. 
\end{thm}

Theorem \ref{thm:main} for the $p$-adic case was proved by C. Nien in \cite{MR2709083} for the Ginzburg-Rallis triple $(\GL_3(\BD),S_\BD,\chi_{S_\BD})$ when $\BD$ is not $F$-split, while for the case that $\BD$ is $F$-split, her proof was later found incomplete. Theorem \ref{thm:main} for the Archimedean case was proved in general by D. Jiang, B. Sun and C. Zhu in \cite{MR2763736}. It remains highly desirable to extend the proof of the local uniqueness of the Ginzburg-Rallis models in \cite{MR2763736} to the case when $\BD$ is $F$-split with $F$ being a $p$-adic local field and to completely prove Theorem \ref{thm:main} over all local fields of characteristic zero. 
We refer to \cite{MR2402684} and a series of papers by C. Wan (\cite{ MR3693584, MR3768394, MR3912931, WAN201974, MR3955540}) for discussions on the other aspects of the local theory for the Ginzburg-Rallis models.  

Let us first recall the strategy of the proof in \cite{MR2763736} for Theorem \ref{thm:main} over the Archimedean local fields. Then we explain what we need in order to obtain a proof of Theorem \ref{thm:main} over a $p$-adic local field $F$ for the case that the quaternion algebra $\BD$ is $F$-split.

When the quaternion algebra $\BD$ is $F$-split, we must have that $G=\GL_3(\BD)=\GL_6(F)$. 
In \cite{MR2763736}, the authors first decompose $G=\GL_6(F)$ into $21$ $P_{2^3}\times P_{2^3}^{\tau}$-orbits, where $\tau$ denotes the transpose, and make use of the 
{\bf transversality} of certain vector fields, a technique developed by \cite{MR348047}, to focus their attention on some open submanifolds. Then the authors of \cite{MR2763736} 
developed two new notions that are called {\bf metrical properness} and {\bf unipotent $\chi$-incompatibility} associated with a descent argument. The descent arguments 
lead to the Multiplicity One Problem for two kinds of models: the trilinear model for $\GL_2$ and the model associated with the pair $(\GL_2,\GL_1)$. 
For those two models, the local uniqueness was concluded by means of the {\bf oscillator representation} (see \cite{MR979944}) in \cite{MR2763736}.

It seems clear that the arguments used in the proof of Theorem \ref{thm:main} for the Archimedean case in \cite{MR2763736} can not be directly extended to the $p$-adic case. 
The unipotent $\chi$-incompatibility for the $p$-adic case will be explained in Remark \ref{rem .2}. However, the transversal property due to \cite{MR348047} and the metrical properness in \cite{MR2763736} are not available in the $p$-adic case, although they are powerful in the Archimedean case. 

On the other hand, in the $p$-adic case, the Bernstein center $\mathfrak{z}=\Fz(G)$ of $G$ plays a similar role to the center of the universal enveloping algebra of 
the Lie algebra $\mathfrak{g}$ of $G$, which can be identified with invariant differential operators on $G$. What used in \cite{MR2763736} is  essentially related to the relevant Laplace operators 
in various descent stages, which lie in the center of the universal enveloping algebra of $\mathfrak{g}$. Following some recent developments in \cite{MR3038552} and \cite{MR3406530}, we are going to use the singular support and the wavefront set of $\mathfrak{z}$-finite distributions in the $p$-adic setting. 
Such an approach has successfully been applied to the proof of the Multiplicity One Theorem in some important cases (see \cite{MR3038552} and \cite{MR3417683}, for instance). In this sense, this paper provides an additional interesting instance of this approach. 

It is important to point out that there is a unitary group version of the Ginzburg-Rallis models, which appeared in the work of L. Zhang (\cite{zhang2019exterior}) and has been studied by C. Wan and L. Zhang in \cite{wan2023multiplicity}. 
It will not be surprising that the method used in \cite{MR2763736} and in this paper can be extended to prove the local uniqueness of the unitary group version of the Ginzburg-Rallis models over all local fields 
of characteristic zero. From the recent classification of strongly tempered spherical pairs by Wan and Zhang in \cite{wan2022periods} and \cite{wan2022multiplicities}, the Ginzburg-Rallis model and its unitary group version 
are among the classification list of 10 strongly tempered spherical pairs. It is expected that the method as displayed in \cite{MR2763736} and in this paper could be extended to treat the Multiplicity One Problem for the other cases in the Wan-Zhang classification list of the strongly tempered spherical pairs. 

The main result of this paper is to prove Theorem \ref{thm:main} for the case that $\BD$ is $F$-split with $F$ $p$-adic, and hence complete the proof of the Multiplicity One Theorem for the 
Ginzburg-Rallis models over all local fields of characteristic zero. 

\begin{thm}\label{mthm}
    Theorem \ref{thm:main} holds when the quaternion algebra $\BD$ is $F$-split, where $F$ is a $p$-adic local field of characteristic zero. 
\end{thm}

The organization of this paper is briefly explained as follows. 
In Section \ref{section 2} we introduce basic notations and conventions of the paper. We will also discuss some general theories about distributions in the $p$-adic set-up as well as some criteria for the vanishing of equivariant distributions including the Gelfand-Kazhdan Construtability Theorem, the Bernstein Localization Principle, and the refined structure of wavefront sets of $\Fz$-fincite distributions 
with $\Fz$ the Bernstein center of the group. We also state with proof a general form of the Gelfand-Kazhdan criterion, building a bridge between the geometric side of the vanishing of equivariant distributions and multiplicity one results in the representation side. In Section \ref{section 3} we will give a guideline for the proof of the Theorem \ref{mthm}. We will first use the Gelfand-Kazhdan criterion to reduce the proof of Theorem \ref{mthm} to Theorem \ref{prop 3.2}. Then we give a detailed explanation of a stratification and filtrations on the group, which helps us to tackle the support of the distributions part by part. Based on those preparations, we prove Theorem \ref{prop 3.2} in Section \ref{Pf32}, modulo four technical Propositions \ref{0=TG3}, \ref{0=TG4}, \ref{0=TG6}, and \ref{0=TG-6}, which are proved 
in Sections \ref{sec-TG3}, \ref{sec-TG4}, \ref{sec-TG6}, and \ref{sec-G6}, respectively. It is worthwhile to point out that in addition to the relatively standard arguments in the proof of Multiplicity 
One Problem, we use the refined structure of the wavefront sets of $\Fz$-finite distributions and their relation to the conormal bundle of the relevant submanifolds. This new approach has been used in four submanifolds: $G^2_{\mathrm{open},1}$, $G^2_{\mathrm{open},2}$, $G^5_{\mathrm{open},1}$ and $G^5_{\mathrm{open},2}$ (see Section \ref{ssec-SG} for definition).  They are paired by certain involutive symmetries  as $(G^2_{\mathrm{open},1}, G^2_{\mathrm{open},2})$ and $(G^5_{\mathrm{open},1}, G^5_{\mathrm{open},2})$. Those two distinguished pairs of submanifolds are treated in detail in Sections \ref{sec-TG3} and \ref{sec-TG6}, respectively.

\section{Distributions and Wave front Sets}\label{section 2}

We get some preparation on general distributions over $p$-adic manifolds, the singular supports for equivariant distributions, and wave front sets of $\Fz$-finite distributions. Let $F$ be a $p$-adic local field of characteristic zero. Let $\RM_n(F)$ be the space of all $n\times n$-matrices, whose elements are written as $x=(x_{ij})$ with $x_{ij}\in F$.

\subsection{Distributions on $l$-spaces}

Recall from \cite{MR0425030} that by an $l$-space $X$ we mean a Hausdorff, locally compact, and zero-dimensional topological space. We denote by $\D(X)$ the space of locally constant complex-valued functions with compact support on $X$, and by $\D^{'}(X)$ the space of linear functionals on $\D(X)$. We equip $\D^{\prime}(X)$ with the weak $*$-topology (see \cite{MR0487295}).

By an $l$-group we mean a topological group whose underlying topological space is an $l$-space. If an $l$-group $G$ acts on an $l$-space $X$, then $G$ acts naturally on $\D(X)$ and $\D^{\prime}(X)$ (\cite[Section 1.7]{MR0425030}). In particular, if $X$ is an $l$-group and $G$ is a subgroup of $X$, for any $g\in G$, we denote by $L_{g}$ and $R_{g}$ 
the left translation and the right translation, respectively.

If $\chi$ is a continuous quasi-character of $G$, we denote
\begin{align*}
\D^{\prime}_{\chi}(X):=\left\{  T\in \D^{\prime}(X)  \colon gT=\chi^{-1}(g)T, \; \forall g\in G      \right\}
\end{align*}
to be the set of $(G,\chi)$-equivariant distributions on $X$ (the inverse here is to make our notations compatible with those in \cite{MR2763736}).
If $U \subset X$ is an open subset, let $Z= X \setminus U$ and recall from \cite{MR0425030} that we have the following exact sequences:
\begin{align}
0\rightarrow\D(U)\rightarrow\D(X)\rightarrow\D(Z)\rightarrow 0 
\end{align}
and
\begin{align}\label{2.2}
0\rightarrow \D^{\prime}(Z)\rightarrow \D^{\prime}(X)\rightarrow \D^{\prime}(U)\rightarrow 0.
\end{align}
If moreover, $U$ and $Z$ are $G$-stable, then we have
\begin{align}\label{2.4}
    0\rightarrow \D_{\chi}^{\prime}(Z)\rightarrow \D^{\prime}_{\chi}(X)\rightarrow \D^{\prime}_{\chi}(U),
\end{align}
since taking the equivariant subspace is left exact in general.
Moreover, we have the following result:

\begin{lem}\label{lem 2.1}
    Let $X=U\cup Y$ be an $l$-space with $U$ being an open subset of $X$ and $Y$ locally closed. Suppose that $G$ acts on $X$ with $U$ and $Y$ being $G$-stable and $\chi$ is a character of $G$. If 
    \begin{align*}
        \D_{\chi}^{\prime}(U)=\D_{\chi}^{\prime}(Y)=0,
    \end{align*}
    then
    $\D_{\chi}^{\prime}(X)=0.$
    \end{lem}

\begin{proof}
    According to the exact sequence \ref{2.4}, any $(G,\chi)$-equivariant distribution on $X$ is supported on $X\setminus U=Y\setminus Y\cap U$. Then we can apply the exact sequence \ref{2.4} again to
    \begin{align*}
        0\rightarrow \D_{\chi}^{\prime}( Y\setminus Y\cap U)\rightarrow \D^{\prime}_{\chi}(Y)\rightarrow \D^{\prime}_{\chi}( Y\cap U),
    \end{align*}
we obtain that 
$\D_{\chi}^{\prime}( Y\setminus Y\cap U)=0$ since $\D^{\prime}_{\chi}(Y)=0$.
\end{proof}

We have another useful lemma:
\begin{lem}\label{lem-prod}
    Suppose that $X$ and $Y$ are two $l$-spaces with $G$-actions. Let $\chi$ be a character of $G$. If $G$ acts trivially on $Y$ and $\mathcal{D}^{\prime}_{\chi}(X)=0$, then $\mathcal{D}^{\prime}_{\chi}(X\times Y)=0$.
\end{lem}

\begin{proof}
    According to \cite[Section 1.22]{MR0425030}, we know $\D(X\times Y)=\D(X)\otimes \D(Y)$, so we have a canonical isomorphism between vector spaces
    \begin{align*}
        \D^{\prime}(X\times Y)=\Hom_{\BC}(\D(X\times Y),\BC)=\Hom_{\BC}(\D(X)\otimes\D(Y),\BC)=\Hom_{\BC}(\D(Y),\D^{\prime}(X)).
    \end{align*}
    By our assumption that $G$ acts trivially on $Y$, $\D_{\chi}^{\prime}(X\times Y)$ can be identified with $\Hom_{\BC}(\D(Y),\D^{\prime}_{\chi}(X))$ via the above isomorphism, which is zero since $\D^{\prime}_{\chi}(X)=0$.
\end{proof}

We recall here some classical criteria for the vanishing of equivariant distributions. The following is \cite[Theorem 6.9]{MR0425030}. 
\begin{thm}[Gelfand-Kazhdan Constructibability]\label{thm 2.2}
    Suppose that the action of an $l$-group $G$ on an $l$-sheaf $(X,\Ff)$ is defined. We assume that
    \begin{itemize}
        \item [(a)] the action of $G$ on $X$ is constructive; and 
        \item [(b)] there are no non-zero $G$-invariant $\Ff$-distributions on any $G$-orbit in $X$.
    \end{itemize}
    Then there are no non-zero $G$-invariant $\Ff$-distributions on $X$.
\end{thm}

\begin{rem}\label{rem 2.3}
We have some comments on Theorem \ref{thm 2.2} in order. 
    \begin{itemize}
        \item [(1)] According to \cite[Theorem A]{MR0425030}, the induced action of an algebraic group on an algebraic variety defined over a $p$-adic local field at the level of rational points is always constructive.
        \item[(2)] If $\chi$ is a character of $G$, and if we take the $l$-sheaf $\Ff$ to be the sheaf of smooth functions on $X$ with the action of $G$ on $\Ff$ given by
        \begin{align*}
            (g.f)(x):=\chi(g^{-1})f(g^{-1}x), \quad g\in G,x\in X,f\in \mathcal{D}(X).
        \end{align*}
        Then the $G$-invariant $\Ff$-distributions are exactly the $(G,\chi)$-equivariant distributions.
    \end{itemize}
\end{rem}

Another criterion for vanishing of equivariant distributions follows from Proposition 2.29, Theorem 6.9, and Theorem A in \cite{MR0425030}.

\begin{thm}\label{thm 2.1}
    Let an algebraic group $\mathbb{G}$ act on an algebraic variety $\mathbb{X}$, both defined over $F$. Write $G=\mathbb{G}(F)$ and $X=\mathbb{X}(F)$. Let $\chi$ be a character of $G$. Suppose that for any $x\in X$, we have
    \begin{align*}
        \chi_{G_x}^{-1}\neq \Delta_{G}|_{G_x}\cdot \Delta_{G_x}^{-1},
    \end{align*}
where $G_x$ is the stabilizer of $x$, $\chi_{G_x}$ is the restriction of $\chi$ to $G_x$, $\Delta_{G}$ and $\Delta_{G_x}$ are the modular functions of the groups $G$ and $G_x$ respectively. Then $\mathcal{D}^{\prime}_{\chi}(X)=0$.
\end{thm}

\begin{rem}\label{rem .2}
If $G$ is unipotent, so is $G_x$. In this case, the modular characters are both trivial. Hence all of the equivariant distributions will vanish if $\chi_{G_x}$ is nontrivial, as a special case of Theorem 
\ref{thm 2.1}. On the other hand, the assertion for a unipotent group $G$ can be viewed as the $p$-adic version of the unipotent $\chi$-incompatibility in the Archimedean case as in \cite[Lemma 3.4]{MR2763736}. Thus the same calculations in \cite[Lemma 5.4, Lemma 5.6, Lemma 6.2, Lemma 6.4, Lemma 6.6 and Lemma 8.4]{MR2763736} also give the same vanishing results in the $p$-adic case.
\end{rem}

To state the next criterion, we recall the following setting. Let $G$ be an $l$-group, and $\tau:G\rightarrow G$ be an anti-involution of $G$. Let $S$ be a closed subgroup and $\chi_S$ a character of $S$. Write $H=S\times S$ and $\chi=\chi_S\otimes\chi_S$. Let $H$ act on $G$ by
\begin{align*}
    (s_1,s_2)x=s_1x\tau(s_2).
\end{align*}

Let $\mathcal{O}$ be an orbit in $G$ under the action of $H$. Let $g_0\in\mathcal{O}$. The following is an easy, but useful statement, which may be known to experts. 
For completeness, we state here for convenience and give a proof. 
\begin{lem}\label{lem 2.7}
    If there exists some $h_0 \in H$ such that $h_0 g_0=\tau(g_0)$ and $\chi(h_0)=1$, then any $(H,\chi)$-equivariant distribution on $\mathcal{O}$ is also $\tau$-invariant.
\end{lem}

\begin{proof}
    Let $T$ be an $(H,\chi)$-equivariant distribution on $\CO$. If $T$ is not $\tau$-invariant then $T^{\prime}:=T-\tau(T)$ is non-zero and is $(H,\chi)$-equivariant as well. Note that $\tau(T^{\prime})=-T^{\prime}$. According to \cite[Proposition 2.29]{MR0425030}, $T^{\prime}$ is proportional to the distribution $T_0$, which is defined as
    \begin{align*}
        T_0(f):=\int_{H/H_{g_0}}f(h)\chi(h) dh,
    \end{align*}
where we first identify $H/H_0$ with $\CO$ via $hH_{g_0}\mapsto hg_0$ and then view $f\in\D(\CO)=\D(H/H_{g_0})$ as a locally constant function on $H$ with compact support module $H_{g_0}$. Then $h\mapsto f(h)\chi(h)$ is a function on $H$ satisfying the condition of \cite[Theorem 1.21]{MR0425030} because of Theorem \ref{thm 2.1} and the measure $dh$ is the $H$-invariant functional in \cite[Theorem 1.21]{MR0425030}.

The involution $\tau$ on $\CO$ induces the well-defined map 
\[
H/H_{g_0}\rightarrow H/H_{g_0}\quad {\rm with}\quad hH_{h_0}\mapsto \tau(h)h_0H_{g_0}.
\]
 We deduce that 
\begin{align*}
    \tau T_0(f)&=\int_{H/H_{g_0}}f(\tau(h)h_0)\chi(\tau(h)h_0)dh
    =c\int_{H/H_{g_0}}f(h_0k)\chi(h_0k)dk=c\int_{H/H_{g_0}}f(k)\chi (k)dk
    =cT_0(f)
\end{align*}
where we make a change of variable $H/H_{g_0}\rightarrow H/H_{g_0}$ by sending $h$ to $k=h_0^{-1}\tau(h)h_0$ and $c$ is its modulus function, which is positive. On the other hand, the involution 
$\tau$ has the property that $\tau T_0(f)=-T_0(f)$. This leads to a contradiction.
\end{proof}

Last but not least, we recall from \cite{MR748505} the Bernstein localization principle.

Let $q: X\rightarrow Y$ be a continuous map between $l$-spaces. Then $\D(X)$ and hence $\D^{\prime}(X)$ become $\D(Y)$-modules. For any $y\in Y$ consider the fiber $X_y:=q^{-1}(\{ y\})$ and identify the space $\D^{\prime}(X_y)$ with the subspace of $\D^{\prime}(X)$ of distributions concentrated on this fiber by using the exact sequence in \eqref{2.2}.

\begin{thm}[Bernstein Localization Principle]
Let $W$ be a closed subspace of $\D^{\prime}(X)$ with a $\D(Y)$-submodule structure. Then $W$ is generated by distributions concentrated on fibers. Equivalently, the sum of subspaces $W^y:=W\cap \D^{\prime}(X_y)$ for all $y\in Y$ is dense in $W$.
\end{thm}

The following direct corollaries are very useful.
\begin{cor}\label{cor 2.9}
Let $G$ act on the space $X$ preserving each fiber $X_y$, and $\chi$ be a quasi-character of $G$, then the space $\D^{\prime}_{\chi}(X)$ is topologically generated by $\D^{\prime}_{\chi}(X)^y$, which can be canonically identified with $\{T\in \D^{\prime}(X_y) \; | \;gT=\chi(g)T \}$.
\end{cor}

\begin{cor}\label{cor 2.10}
Assumptions are as in Corollary \ref{cor 2.9}. If $\D^{\prime}_{\chi}(X)=0$, then $\D^{\prime}_{\chi}(Z)=0$ 
for any $G$-invariant locally closed subspace $Z$, which is a union of fibers, of $X$.
\end{cor}

\subsection{Wavefront Sets}
Before we state results about wavefront sets used in our proof, we first recall some standard material about $p$-adic analytic manifolds.
Let $X$ be a $p$-adic analytic manifold (see \cite{MR0176987}). Define a Schwartz measure as a locally constant section with compact support of the density bundle. 
Denote by $C^{-\infty}(X)$ the space of generalized functions, which are defined as elements in the dual space of Schwartz measures on $X$.
As \cite[Proposition 16.37]{MR2954043}, there exists a smooth density $\omega_X$ with full support on $X$. By using such a density, one may identify Schwartz functions with Schwartz measures in the way that 
takes $f$ to $f\omega_X$. Hence we can identify distributions with generalized functions.

Let $\varphi: X\rightarrow Y$ be a surjective submersion. The pushing forward associated with $\varphi$ of Schwartz measures induces the pulling back of generalized functions:
\begin{align*}
    \varphi^*:C^{-\infty}(Y)\rightarrow C^{-\infty}(X),
\end{align*}
which is an injection.
Recall from \cite[Section 2.2]{MR2763736}, we have the notions such as local slice, slice, and relative stability for the $p$-adic setting. In particular, Lemmas 2.5, 2.6, and 2.7 there are still available here. For convenience, we discuss them briefly here.

Let $M$ be a $p$-adic analytic manifold with an action of a $p$-adic analytic group $H$, let $\mathfrak{M}$ be a submanifold of $M$ and denote 
\begin{align*}
    \rho_{\mathfrak{M}}:H\times \mathfrak{M}\rightarrow M
\end{align*}
as the action map.

\begin{defn}\label{defn-1}
Let $\mathfrak{M}$ be a submanifold of $M$.
    \begin{itemize}
        \item [(1)] We say $\FM$ is a local slice of $M$ if $\rho_{\FM}$ is a submersion and an $H$ slice of $\rho_{\FM}$ is a surjective submersion.
        \item [(2)] Given two submanifolds $\FZ\subset\FM$ of $M$, we say that $\FZ$ is relatively $H$ stable in $\FM$ if 
        \begin{align*}
            \FM\cap H\FZ=\FZ.
        \end{align*}
    \end{itemize}
\end{defn}

\cite[Lemma 2.6]{MR2763736} also holds in the $p$-adic setting:
\begin{lem}
    Let $\FM$ be an $H$ slice of $M$, and let $\FZ$ be a relatively $H$ stable submanifold of $\FM$. Then $Z=H\FZ$ is a submanifold of $M$, and $\FZ$ is an $H$ slice of $Z$. Furthermore, if $\FZ$ is closed in $\FM$, then $Z$ is closed in $M$.
\end{lem}

We assume that $\mathfrak{M}$ is a local $H$-slice of $M$, that is, $\rho_{\mathfrak{M}}$ is a submersion. Let $H_{\mathfrak{M}}$ be a closed subgroup of $H$ that leaves $\mathfrak{M}$ stable. Let $H$ act $H\times \mathfrak{M}$ by left multiplication on the first factor, and $H_{\mathfrak{M}}$ act on $H\times \mathfrak{\M}$ by
\begin{align*}
    g(h,x)=(ghg^{-1},gx),\quad g\in H_{\mathfrak{M}}, h\in H,x\in\mathfrak{M}.
\end{align*}
Then the submersion $\rho_{\mathfrak{M}}$ is $H$-equivariant as well as $H_{\mathfrak{M}}$-equivariant. Therefore the pulling back yields a linear map (we need to fix smooth densities on $H\times\mathfrak{M}$ as well as on $M$)
\begin{align*}
    \rho_{\mathfrak{M}}^*:\mathcal{D}^{'}_{\chi}(M)\rightarrow \mathcal{D}^{'}_{\chi}(H\times\mathfrak{M})\cap\mathcal{D}^{'}_{\chi_{\mathfrak{M}}}(H\times\mathfrak{M}),
\end{align*}
where $\chi_{\mathfrak{M}}:=\chi|_{H_{\mathfrak{M}}}$. By the Schwartz Kernel Theorem and the uniqueness of the Haar measure, we have
\begin{align*}
    \mathcal{D}^{'}_{\chi}(H\times\mathfrak{M})=\chi\otimes\mathcal{D}^{'}(\mathfrak{M}).
\end{align*}
Consequently, we obtain that 
\begin{align*}
    \mathcal{D}^{'}_{\chi}(H\times\mathfrak{M})\cap\mathcal{D}^{'}_{\chi_{\mathfrak{M}}}(H\times\mathfrak{M})=\chi\otimes\mathcal{D}^{'}_{\chi_{\mathfrak{M}}}(\mathfrak{M}).
\end{align*}
We record this as a lemma, which the $p$-adic version of \cite[Lemma 2.7]{MR2763736}.
\begin{lem}\label{lem 2.2}
    There is a well-defined map, which is called the restriction to $\mathfrak{M}$,
    \begin{align*}
        \mathcal{D}^{'}_{\chi}(M)\rightarrow \mathcal{D}^{'}_{\chi_{\mathfrak{M}}}(\mathfrak{M}), \quad f\mapsto f|_{\mathfrak{M}}
    \end{align*}
    by requiring that 
    \begin{align*}
        \rho_{\mathfrak{M}}^*(f)=\chi\otimes f|_{\mathfrak{M}}.
    \end{align*}
The map is injective when $\mathfrak{M}$ is an $H$-slice, i.e., $\rho_{\mathfrak{M}}$ is a surjective submersion.
\end{lem}

Now we discuss briefly the wavefront sets of $\Fz$-finite distributions and state the results that are useful in our proof. As before, if $\BG$ is a reductive algebraic group defined over $F$, we write $G=\BG(F)$ for the $F$-rational points. Let $\mathfrak{z}=\Fz(G)$ be the Bernstein center of $G$. 
It is well-known that the Bernstein center $\mathfrak{z}$ can be realized as the space of essentially compact $G$-invariant distributions under the adjoint action, which is denoted by $\mathfrak{z}=\D^{\prime}(G)^{G}_{e.c}$ and is an algebra with the convolution as its multiplication. It is clear that the Bernstein center of $\BG(F)$ for $p$-adic local fields $F$ plays an analogous role as the center of the universal enveloping algebra of the Lie algebra of $\BG(F)$ for Archimedean local fields $F$. 

In the Archimedean case, the center $\mathfrak{z}$ of the universal enveloping algebra of the Lie algebra of $G$ plays an important role in the study of the Multiplicity One Problem (see \cite{MR348047} for the Whittaker models and \cite{MR2763736} for the Ginzburg-Rallis models). In the $p$-adic case, distributions arising in representation theory 
are often $\mathfrak{z}$-finite. As indicated in the recent work (\cite{MR3038552} and \cite{MR3406530}), it is natural to consider the $\Fz$-finite distributions in the study 
of the Multiplicity One Problem and other relevant topics. 
In this paper, we shall deal with $\mathfrak{z}$-finite distributions arising in the Ginzburg-Rallis model over a $p$-adic local field $F$ of characteristic zero.

We start with  a brief review of the theory of wavefront sets over a $p$-adic local field $F$ following the developments in \cite{MR2631841} and \cite{MR3417683}.
We first define the wavefront sets for vector spaces.

\begin{defn}
Let $V$ be a finite-dimensional vector space over $F$ and $V^*$ be the linear dual. 
\begin{itemize}
    \item [(1)]  Let $f\in C^{\infty}(V^*)$ and $w_0\in V^{*}$. We say that $f$ vanishes asymptotically in the direction of $w_0$ if there is some $\rho\in\mathcal{D}(V^{*})$ such that $\rho(w_0)\neq 0$ and the function on $V^{*}\times F$ defined by 
    \[
    (w,\lambda)\mapsto f(\lambda w)\rho(w)
    \]
    is compactly supported.    
    \item [(2)] Let $U\subset V$ be an open set and $T\in\mathcal{D}^{'}(U)$. Let $x_0\in U$ and $w_0\in V^{*}$. We say that $T$ is smooth at $(x_0,w_0)$ if there exists a compactly supported non-negative function $\rho\in\mathcal{D}(V)$ with $\rho(x_0)\neq0$ such that the Fourier transform $\CF(\rho T)$ vanishes asymptotically in the direction $w_0$.
    \item[(3)] The complement in the cotangent bundle $T^{*}U$ of the set of smooth pairs $(x_0,w_0)$ of $T$ is the wavefront set of $S$ and is denoted by $\wf(T)$.
    \item[(4)] For a point $x\in U$ we denote $\wf_x(T):=\wf(T)\cap T^{*}_x(U)$.
\end{itemize}
\end{defn}

Note that the notion of the Fourier transform $\CF$ on a vector space depends on a choice of a non-trivial character of $F$, however, this dependence affects the Fourier transform only by dilation, and thus does not change the wave front set.

We recall Corollary 2.4 in \cite{MR3417683} and state it as a lemma.
\begin{lem}\label{wf-diff}
    Let $U\subset F^m$ and $V\subset F^n$ be open subsets and $\phi:U\rightarrow V$ be an analytic diffeomorphism. Then for any $T\in\mathcal{D}^{'}(U)$, the wavefront sets enjoy 
    the property:
    \begin{align*}
        \wf(\phi^*(T))=\phi^*(\wf(T)).
    \end{align*}
\end{lem}

It is clear that Lemma \ref{wf-diff} leads to the definition of the wavefront set of any distribution on any analytic manifold over $F$, as a subset of the cotangent bundle.

\begin{defn}
    Let $X$ be an $F$-analytic manifold an $T\in\mathcal{D}^{'}(X)$. The wavefront set $\wf(T)$ is defined to be the set of all $(x,\lambda)\in T^*X$ which lie in the wavefront set of $T$ in some local coordinates. In other words, $(x,\lambda)\in \wf(T)$ if there exists an open neighborhood $U$ of $x$ in $X$ and $V\subset F^n$, an analytic diffeomorphism $\phi:U\rightarrow V$ and $(y,\beta)\in T^*V$ such that $\phi(x)=y$, $d_x\phi^*(\beta)=\lambda$ and $(y,\beta)\in \wf((\phi^{-1})^*(T|_U))$.
\end{defn}

A local property of the wavefront sets, which enables us to keep information about the wavefront set of the restriction of a distribution to an open submanifold, is 
Proposition 2.1 in \cite{MR2631841}. We state it here without recalling its proof. 

\begin{prop}\label{prop 2.14}
    Let $U\subset X$ be an open submanifold, then for any $T\in\mathcal{D}^{'}(X)$, we have
    \begin{align*}
        \wf(T)_x=\wf(T|_U)_x
    \end{align*}
    for any $x\in U$.
\end{prop}

We recall from \cite{MR3038552} and \cite{MR3406530} some useful facts on the wavefront sets, which will play important roles in the proof of this paper. 
Again, let $\mathbb{G}$ be a reductive group over $F$, and $G=\mathbb{G}(F)$ be its $F$-rational points equipped with the analytic topology.
The first is \cite[Theorem A]{MR3406530} about the wavefront set of $\mathfrak{z}$-finite distributions. 

\begin{thm}\label{thm 2.15}
    Let $T\in\mathcal{D}^{'}(G)$ be a $\mathfrak{z}$-finite distribution. Then for any $x\in G$ we have
    \begin{align*}
        \wf_x(T)\subset\mathcal{N},
    \end{align*}
    where $\mathcal{N}$ is the nilpotent cone of the dual $\Fg^*$ of the Lie algebra $\Fg$ of $G$, and we identify the Lie algebra $\Fg$ with $T_x(G)$ using the right action.
\end{thm}

Another useful fact is \cite[Theorem 4.1.2]{MR3038552} on the invariance of the wavefront sets with respect to the shift of the conormal bundle if the distribution is supported on a submanifold. 

\begin{thm}\label{thm 2.12}
    Let $Y\subset X$ be $F$-analytic manifolds and let $y\in Y$. Let $T\in\mathcal{D}^{'}(X)$ and suppose $\mathrm{supp}(T)\subset Y$. Then $\wf_y(T)$ is invariant with respect to the shifts by the conormal vector space $\cn^X_{Y,y}$.
\end{thm}

We can combine the above two theorems to obtain the following criterion:

\begin{prop}\label{prop 2.7}
    Let $G$ be the $F$-rational points of a reductive group over $F$ and $T$ be a $\mathfrak{z}$-finite distribution on $G$, where $\mathfrak{z}$ is the Bernstein center of $G$. 
    Let $M$ be an open submanifold of $G$ and $N$ a submanifold of $M$. If the restriction of $T$ on $M$ is supported on $N$ and for some point $x\in \mathrm{supp}(T)$ and there is a non-zero element in the conormal fiber $\cn^M_{N,x}$ that is not in the nilpotent cone $\mathcal{N}$ of $\mathfrak{g}^*$, where we identify the Lie algebra $\mathfrak{g}$ with the tangent space at $x$ of $G$ using the right action, then $T$ is supported in the complement of $M$ in $G$.
\end{prop}

\begin{proof}
    According to Theorem $A$ of \cite{MR3406530}, for any $x$ in $G$, the wavefront set of $T$ at $x$ is contained in $\mathcal{N}$, so is the wavefront set of $T|_{M}$ since the restriction to an open submanifold does not change the fiber of the wavefront set. Then due to Theorem 4.1.2 in \cite{MR3038552}, the wavefront set is invariant under the shift by the conormal space. If $T|_{M}\neq 0$ and $x\in\supp(T)$, then the point $(x,0)$ belongs to the wavefront set.  Hence the whole set $\cn^M_{N,x}$ is contained in 
    the wavefront set, which should then further be contained in the nilpotent cone $\mathcal{N}$. This leads a contradiction.
\end{proof}

\subsection{A general form of Gelfand-Kazhdan criterion}

In this section, we present a general form of the Gelfand-Kazhdan criterion over $p$-adic local fields in terms of distributions, which can be viewed as a counterpart of the more complicated one for Archimedean case as in \cite{MR2820401}.

\begin{thm}\label{thm 2.3}
Let $H_1$ and $H_2$ be two closed subgroups of a unimodular $l$-group $G$ with continuous characters
\begin{align*}
\chi_{H_i}:H_i\rightarrow \mathbb{C}^{\times}, \quad i=1,2.
\end{align*}
Let $\sigma$ be a continuous anti-involution of $G$. Assume that for every $T\in(\D^{\prime}(G))^{\infty}$, which is also an eigenvector of the Bernstein center $\mathfrak{z}=\D^{\prime}(G)^{G}_{e.c}$, the quasi-invariant properties: 
\begin{align*}
L_hT=\chi_{H_1}^{-1}(h)T, \quad s\in H_1,
\quad {\rm and}\quad 
R_hT=\chi_{H_2}^{-1}(h)T,\quad s\in H_2,
\end{align*}
implies that $\sigma T=T$. Then for any irreducible admissible representation $(\pi,V)$ of $G$, one has
\begin{align*}
\mathrm{dim}\;\mathrm{Hom}_{H_1}(\pi,\chi_{H_1})\times\mathrm{dim}\;\mathrm{Hom}_{H_2}(\pi^{\vee},\chi_{H_2})\leq 1,
\end{align*}
where $(\pi^{\vee},V^{\vee})$ is the contragredient of $\pi$.
\end{thm}

\begin{proof}
The argument is standard (see \cite{MR2820401}). Suppose that both $\mathrm{Hom}_{H_1}(\pi,\chi_{H_1})$ and $\mathrm{Hom}_{H_2}(\pi^{\vee},\chi_{H_2})$ are nonzero. Take
\begin{align*}
0\neq v\in \mathrm{Hom}_{H_2}(\pi^{\vee},\chi_{H_2})\subset V
\quad {\rm and}
\quad 
0\neq v^{\vee} \in \mathrm{Hom}_{H_2}(\pi,\chi_{H_1})\subset V^{\vee}.
\end{align*}
Consider the matrix coefficient $\varphi_{v^{\vee},v}(g)=\langle \pi(g)v,v^{\vee} \rangle \in C^{\infty}(G)\hookrightarrow \D^{\prime}(G)$ by a fixed a Haar measure $dg$ on $G$. It is clear as functions on $G$, we have
\begin{itemize}
	\item [(1)] $\varphi_{v^{\vee},v}(hg)=\chi_{H_1}(h)\varphi_{v^{\vee},v}(g), \forall h\in H_1,$
	\item [(2)] $\varphi_{v^{\vee},v}(gh)=\chi_{H_2}^{-1}(h)\varphi_{v^{\vee},v}(g), \forall h\in H_2.$
\end{itemize}
Hence as a distribution, we have
\begin{itemize}
\item [(1)] $L_h \varphi_{v^{\vee},v}=\chi_{H_1}^{-1}(h)\varphi_{v^{\vee},v}, \forall h\in H_1, $
\item [(2)] $R_h\varphi_{v^{\vee},v}=\chi_{H_2}^{-1}(h)\varphi_{v^{\vee},v},\forall h\in H_2.$
\end{itemize}
Then we obtain that 
\begin{align}\label{2.3}
\sigma \varphi_{v^{\vee},v}=\varphi_{v^{\vee},v}
\end{align}
 as distributions since it is clear that $\varphi_{v^{\vee},v}$ is an eigenvector for the Bernstein center.

We claim that for all $f\in \mathcal{D}(G)$, $fv=0$ if and only if $(\sigma f)^{\vee}v^{\vee}=0$, where the action of $f$ is given via the above fixed Haar measure $dg$ and $f^{\vee}(g):=f(g^{-1})$ for functions $f$.

In fact, by using the irreducibility of $\pi$, we deduce that $fv=0$ if and only if
\begin{align*}
\langle g(fv),v^{\vee}\rangle=0,\forall g\in G,
\quad {\rm i.e.}\quad 
\langle L_g(f)v,v^{\vee}\rangle=0,\forall g\in G.
\end{align*}
As a consequence of (\ref{2.3}), this is equivalent to $\langle \sigma(L_g f )v,v^{\vee}\rangle=0 $.
Now the claim follows since
\begin{align*}
\langle \sigma(L_g f )v,v^{\vee}\rangle=\langle \sigma(f)(\sigma(g)v),v^{\vee}\rangle=\langle \sigma(g)v,(\sigma(f))^{\vee}v^{\vee})
\end{align*}

Finally, let $0\neq v^{\prime}\in\mathrm{Hom}_{H_2}(V^{\vee},\chi_{H_2})\subset V $ be anothere element. Apply the above claim twice, we get for all $f\in \D(G)$, $fv=0$ if and only if $fv^{\prime}=0$. Therefore the two $G$-homomorphism $f\mapsto fv$ and $f\mapsto fv^{\prime}$ frow $\D(G)$ to $V$ have the same kernel, say $J$, where the representation on $\D(G)$ is the left regular representation. Using Schur's lemma, we know that $v^{\prime}$ is a scalar of $v$, which proves
\begin{align*}
\mathrm{dim}\;\mathrm{Hom}_{H_1}(V,\chi_{H_1})=1.
\end{align*}
The same argument shows
\begin{align*}
\mathrm{dim}\;\mathrm{Hom}_{H_2}(V^{\vee},\chi_{H_2})=1.
\end{align*}
\end{proof}

\section{Proof of Theorem \ref{mthm}}\label{section 3}

\subsection{Standard argument}
We return to the group $G= \GL_6(F)$. Recall the subgroup $S$ and its character $\chi_{S}$ are as defined in Section \ref{introduction}. As in \cite{MR2763736}, we set \begin{align*}
    H=S \times S \; \mathrm{and} \; \chi= \chi_S \otimes \chi_S.
\end{align*}
Take the anti-involution $\tau$ of $G$ to be the transpose of matrices, and let $H$ act on $G$ by 
\begin{align*}
    (g_1,g_2)x=g_1xg_2^{\tau},
\end{align*}
as in \cite{MR2763736}.

By Theorem \ref{thm 2.3}, in order to prove Theorem \ref{mthm}, it is sufficient to prove the following theorem. 

\begin{thm}\label{thm 3.1}
    Let $T$ be a distribution on $G$, which is an eigenvector of the Bernstein center $\mathfrak{z}=\D^{\prime}(G)^{G}_{e.c.}$. If $T$ satisfies the following quasi-invariant property: 
    \begin{align*}
        (s_1,s_2)T=\chi^{-1}(s_1,s_2)T, \; for \; all \; (s_1,s_2)\in H, \; \mathrm{i.e.} \; T \in \D^{\prime}(G)_{\chi},
    \end{align*}
    then the distribution $T$ is $\tau$-invariant, i.e.,
    $T= T^{\tau}.$
\end{thm}

It will be slightly more convenient to work with the following (twisted by $\tau$) group 
\begin{align*}
    \widetilde{H}=\{1,\tau\}\ltimes H,
\end{align*} 
where the semidirect product is defined by the action
\begin{align*}
    \tau(g_1,g_2)=(g_2,g_1).
\end{align*}
Extend $\chi$ to a character  $\widetilde{\chi}$ of  $\widetilde{H}$ by requiring that 
$\widetilde{\chi}(\tau)=-1,$
and extend the action of $H$ on $G$ to that of $\widetilde{H}$ by putting
$\tau(x)=x^{\tau}.$ 

Let $T \in \mathcal{D}^{\prime}(G)_{\chi}$ be an eigenvector under the action of the Bernstein center $\mathfrak{z}=\D^{\prime}(G)^{G}_{e.c.}$  as in the setting of Theorem \ref{thm 3.1}. Then it is easy to check that 
$T-T^{\tau} \in \D^{\prime}(G)_{\widetilde{\chi}}$. 
One can see easily by definition that $T^{\tau}$ is also an eigenvector under the action of the Bernstein center $\Fz$. It follows that $T-T^{\tau}$ is $\mathfrak{z}$-finite. Hence Theorem \ref{thm 3.1} is equivalent to the following

\begin{thm}\label{prop 3.2}
    If  $T \in \D^{\prime}(G)_{\widetilde{\chi}}$ is a $\mathfrak{z}$-finite distribution, then $T=0$.    
\end{thm}
The proof of Theorem \ref{prop 3.2} will be given in Sections \ref{ssec-SG} and  \ref{Pf32}.

\subsection{Stratification of $G$}\label{ssec-SG}
Let $P=P_{2^3}$ be the standard parabolic subgroup of $G=\GL_6$ as introduced in \eqref{P222}, whose elements are of the form:
\begin{align*}
     \begin{pmatrix}
        a_1 & b & d \\ 0 & a_2 & c \\ 0 & 0 &a_3
    \end{pmatrix} \in G. 
\end{align*}
In order to prove Theorem \ref{prop 3.2}, we have to figure out the $H$-orbits in $G$. Since $S\subset P$, as a first approximation, we follow \cite[Section 4]{MR2763736} to introduce a stratification of $G$ 
with respect to the action of $P\times P^\tau$. 

For $x\in G$, we define its rank-matrix
\begin{align*}
    R(x)=\begin{pmatrix}
        \mathrm{rank}_{4\times 4}(x) & \mathrm{rank}_{4\times 2}(x)\\
        \mathrm{rank}_{2\times 4}(x) & \mathrm{rank}_{2\times 2}(x)
    \end{pmatrix},
\end{align*}
where $\mathrm{rank}_{i\times j}(x)$ is the rank of the lower right $i\times j$ block of $x$. Then $R(x)$ takes the following $21$ possible rank-matrices (see \cite[Section 4]{MR2763736}):

 \small{   \begin{center}
 $   \begin{pmatrix}
        4 & 2 \\ 2 & 2
    \end{pmatrix},\begin{pmatrix}
        4 & 2 \\ 2 & 1
    \end{pmatrix},\begin{pmatrix}
        4 & 2 \\ 2 & 0
    \end{pmatrix},$\end{center}
\begin{center}
   $ \begin{pmatrix}
        3 & 2 \\ 2 & 2
    \end{pmatrix}, \begin{pmatrix}
        3 & 2 \\ 2 & 1
    \end{pmatrix},\begin{pmatrix}
        3 & 2 \\ 1 & 1
    \end{pmatrix},\begin{pmatrix}
        3 & 1 \\ 2 & 1
    \end{pmatrix},\begin{pmatrix}
        3 & 1 \\ 1 & 1
    \end{pmatrix},\begin{pmatrix}
        3 & 2 \\ 1 & 0
    \end{pmatrix},\begin{pmatrix}
        3 & 1 \\ 2 & 0
    \end{pmatrix},\begin{pmatrix}
        3 & 1 \\ 1 & 0
    \end{pmatrix},$\end{center}
    \begin{center}$
    \begin{pmatrix}
        2 & 2 \\ 2 & 2
    \end{pmatrix},\begin{pmatrix}
        2 & 2 \\ 1 & 1
    \end{pmatrix},\begin{pmatrix}
        2 & 1 \\ 2 & 1
    \end{pmatrix},\begin{pmatrix}
        2 & 1 \\ 1 & 1
    \end{pmatrix},\begin{pmatrix}
        2 & 2 \\ 0 & 0
    \end{pmatrix},\begin{pmatrix}
        2 & 0 \\ 2 & 0
    \end{pmatrix},\begin{pmatrix}
        2 & 1 \\ 1 & 0
    \end{pmatrix},\begin{pmatrix}
        2 & 1 \\ 0 &0
    \end{pmatrix},\begin{pmatrix}
        2 & 0 \\ 1 & 0
    \end{pmatrix},\begin{pmatrix}
        2 & 0 \\ 0 & 0
    \end{pmatrix}.$
\end{center}}
For $R$ one of the above rank-matrices, define
\begin{align}\label{GR}
    G_R:=\{x\in G\;|\; R(x)=R \},
\end{align}
which is a $P$-$P^{\tau}$ double coset, hence $H$-invariant. Note that 
\begin{align}\label{GGR}
    G=\bigsqcup_R G_R,
\end{align}
where $R$ runs over the all rank-matrices above.

In the following, we define some open $H$-invariant submanifolds of $G$.
\begin{align}\label{G6op}
    G_{\mathrm{open}}^6:=\bigsqcup_R G_R,
\end{align}
where $R$ runs through the six rank-matrices:
\[ \begin{pmatrix}
        4 & 2 \\ 2 & 2
    \end{pmatrix},\begin{pmatrix}
        4 & 2 \\ 2 & 1
    \end{pmatrix},\begin{pmatrix}
        3 & 2 \\ 2 & 2
    \end{pmatrix},\begin{pmatrix}
        3 & 2 \\ 2 & 1
    \end{pmatrix}, \begin{pmatrix}
        3 & 2 \\ 1 & 1
    \end{pmatrix}, \begin{pmatrix}
        3 & 1 \\ 2 & 1
    \end{pmatrix}.
    \]
Write
\begin{align}\label{G5op1}
    G^5_{\mathrm{open},1}:=\bigsqcup_R G_R,
\end{align}
where $R$ runs from the five rank-matrices:
$    \begin{pmatrix}
        4 & 2 \\ 2 & 2
    \end{pmatrix},\begin{pmatrix}
        4 & 2 \\ 2 & 1
    \end{pmatrix},\begin{pmatrix}
        3 & 2 \\ 2 & 2
    \end{pmatrix},\begin{pmatrix}
        3 & 2 \\ 2 & 1
    \end{pmatrix},\begin{pmatrix}
        3 & 1 \\ 2 & 1
    \end{pmatrix},$
and
\begin{align}\label{G5op2}
    G^5_{\mathrm{open},2}:=\bigsqcup_R G_R,
\end{align}
where $R$ runs from the five rank-matrices:
$  \begin{pmatrix}
        4 & 2 \\ 2 & 2
    \end{pmatrix},\begin{pmatrix}
        4 & 2 \\ 2 & 1
    \end{pmatrix},\begin{pmatrix}
        3 & 2 \\ 2 & 2
    \end{pmatrix},\begin{pmatrix}
        3 & 2 \\ 2 & 1
    \end{pmatrix},\begin{pmatrix}
        3 & 2 \\ 1 & 1
    \end{pmatrix}.$
It is clear that 
\begin{align}\label{G6=G51G52}
    G_{\mathrm{open}}^6=G^5_{\mathrm{open},1}\cup G^5_{\mathrm{open},2}.
\end{align}
We define
\begin{align}\label{G4op}
    G^4_{\mathrm{open}}:=\bigsqcup_R G_R,
\end{align}
where $R$ runs through the four rank-matrices:
$    \begin{pmatrix}
        4 & 2 \\ 2 & 2
    \end{pmatrix},\begin{pmatrix}
        4 & 2 \\ 2 & 1
    \end{pmatrix},\begin{pmatrix}
        3 & 2 \\ 2 & 2
    \end{pmatrix},\begin{pmatrix}
        3 & 2 \\ 2 & 1
    \end{pmatrix};$
and 
\begin{align}\label{G3op}
G^3_{\mathrm{open}}:=\bigsqcup_R G_R,
\end{align}
where $R$ runs through the three rank-matrices:
$    \begin{pmatrix}
        4 & 2 \\ 2 & 2
    \end{pmatrix},\begin{pmatrix}
        4 & 2 \\ 2 & 1
    \end{pmatrix},\begin{pmatrix}
        3 & 2 \\ 2 & 2
    \end{pmatrix}.$
Finally, we define 
\begin{align}\label{G2op1}
    G^2_{\mathrm{open},1}:= \bigsqcup_R G_R,
\end{align}
where $R$ runs through the two rank-matrices:
$    \begin{pmatrix}
        4 & 2 \\ 2 & 2
    \end{pmatrix},\begin{pmatrix}
        3 & 2 \\ 2 & 2
    \end{pmatrix};$
and 
\begin{align}\label{G2op2}
    G^2_{\mathrm{open},2}:= \bigsqcup_R G_R,
\end{align}
where $R$ runs through the two rank-matrices:
$    \begin{pmatrix}
        4 & 2 \\ 2 & 2
    \end{pmatrix},\begin{pmatrix}
        4 & 2 \\ 2 & 1
    \end{pmatrix}.$
It is clear that 
\begin{align}\label{G3=G21G22}
    G^3_{\mathrm{open}}=G^2_{\mathrm{open},1}\cup G^2_{\mathrm{open},2}.
\end{align}
Thus we obtain a sequence of open subsets of $G$:
\begin{align}\label{openfilter}
    G^3_{\mathrm{open}}=G^2_{\mathrm{open},1}\cup G^2_{\mathrm{open},2}\subset G^4_{\mathrm{open}}\subset G^6_{\mathrm{open}}=G^5_{\mathrm{open},1}\cup G^5_{\mathrm{open},2}\subset G,
\end{align}
which indicates the main steps in the proof of Theorem \ref{prop 3.2}.

\subsection{The proof of Theorem \ref{prop 3.2}}\label{Pf32}
We are going to prove Theorem \ref{prop 3.2} based on the sequence of open subsets given in \eqref{openfilter}. 

Assume that $T \in \D^{\prime}(G)_{\widetilde{\chi}}$ is $\mathfrak{z}$-finite as in Theorem \ref{prop 3.2}. We denote by $T|_{G^i_{\mathrm{open}}}\in \mathcal{D}^{\prime}_{\chi}(G^i_{\mathrm{open}})$ the restriction of $T$ to the open $H$-invariant submanifold $G^i_{\mathrm{open}}$ with $i=3,4,6$, and by $T|_{G^k_{\mathrm{open},i}}\in \mathcal{D}^{\prime}_{\chi}(G^k_{\mathrm{open},i})$ the restriction of $T$ to the open $H$-invariant submanifold $G^k_{\mathrm{open},i}$ with $k=2,5$ and $i=1,2$. 

We first prove the following vanishing property. 
\begin{prop}\label{0=TG3}
    If $T \in \D^{\prime}(G)_{\widetilde{\chi}}$ is $\mathfrak{z}$-finite as in Theorem \ref{prop 3.2}, then 
    the following vanishing properties:
    \begin{align}\label{0=TG2}
        T|_{G^2_{\mathrm{open},1}}=0\quad {\rm and}\quad T |_{G^2_{\mathrm{open},2}}=0
    \end{align}
hold, which implies that 
$T |_{G^3_{\mathrm{open}}}=T |_{G^2_{\mathrm{open},1}\cup G^2_{\mathrm{open},2}}=0.$
\end{prop}

We will first prove $T|_{G^2_{\mathrm{open},1}}=0$ in Section \ref{ssec-G21} and $T |_{G^2_{\mathrm{open},2}}=0$ in Section \ref{ssec-G22}. As a consequence of the vanishing properties in \eqref{0=TG2}, by considering the support of the distribution $T|_{G^3_{\mathrm{open}}}$ and make use of the exact sequence in \eqref{2.4}, we obtain that 
$T |_{G^3_{\mathrm{open}}}=T |_{G^2_{\mathrm{open},1}\cup G^2_{\mathrm{open},2}}=0.$ Hence to prove Proposition \ref{0=TG3}, it is enough to prove the vanishing assertion in \eqref{0=TG2}, which will be done 
in Section \ref{sec-TG3}. 
Next we prove in Section \ref{sec-TG4} that
\begin{prop}\label{0=TG4}
    If $T \in \D^{\prime}(G)_{\widetilde{\chi}}$ is $\mathfrak{z}$-finite as in Theorem \ref{prop 3.2}, then the following vanishing 
    \begin{align}\label{G4-3}
    \mathcal{D}^{\prime}_{\widetilde{\chi}}(G^4_{\mathrm{open}} \setminus G^3_{\mathrm{open}})=0
\end{align}
holds, which implies that 
    $T |_{G^4_{\mathrm{open}}}=0.$
\end{prop}
From the vanishing property in \eqref{G4-3} and the exact sequence in \eqref{2.4}, we obtain that 
$T |_{G^4_{\mathrm{open}}}=0.$ Hence it is enough to prove the vanishing assertion in \eqref{G4-3}, which will be done in Section \ref{sec-TG4}. 
With Propositions \ref{0=TG3} and \ref{0=TG4}, we prove in Section \ref{sec-TG6} that 
\begin{prop}\label{0=TG6}
    If $T \in \D^{\prime}(G)_{\widetilde{\chi}}$ is $\mathfrak{z}$-finite as in Theorem \ref{prop 3.2}, then 
    the following vanishing properties:
    \begin{align}\label{0=TG5}
        T|_{G^5_{\mathrm{open},1}}=0\quad {\rm and}\quad T |_{G^5_{\mathrm{open},2}}=0
    \end{align}
hold, which implies that 
$T |_{G^6_{\mathrm{open}}}=T |_{G^5_{\mathrm{open},1}\cup G^5_{\mathrm{open},2}}=0.$
\end{prop}

Finally, we prove in Section \ref{sec-G6} that 
\begin{prop}\label{0=TG-6}
    With the notations as introduced above, the following vanishing 
    \[
    \mathcal{D}^{\prime}_{\widetilde{\chi}}(G \setminus G^6_{\mathrm{open}})=0
\]
holds.
\end{prop}

Theorem \ref{prop 3.2} can finally be proved as follows. If $T \in \D^{\prime}(G)_{\widetilde{\chi}}$ is $\mathfrak{z}$-finite as in Theorem \ref{prop 3.2}, by applying the exact sequence in \eqref{2.4} to 
the situation that $G^6_{\mathrm{open}}\subset G$ is open, and by the property that $\mathcal{D}^{\prime}_{\widetilde{\chi}}(G \setminus G^6_{\mathrm{open}})=0$ in Proposition \ref{0=TG-6}, 
in order to show that $T=0$ as a distribution on $G$, it suffices to show that $T |_{G^6_{\mathrm{open}}}=0$, which follows from Proposition \ref{0=TG6}. We are done.

\section{Proof of Proposition \ref{0=TG3}}\label{sec-TG3}

We prove Proposition \ref{0=TG3}, namely the following vanishing properties:
    \begin{align}\label{0=TG2-12}
        T|_{G^2_{\mathrm{open},1}}=0\quad {\rm and}\quad T |_{G^2_{\mathrm{open},2}}=0
    \end{align}
for any $T \in \D^{\prime}(G)_{\widetilde{\chi}}$, which is $\mathfrak{z}$-finite as in Theorem \ref{prop 3.2}. 

\subsection{The case of $G^2_{\mathrm{open},1}$}\label{ssec-G21}
In order to study the open subset $G^2_{\mathrm{open},1}$ of $G$, we introduce the following stratification. 
It is straightforward to check that 
\begin{align*}
    G^2_{\mathrm{open},1} \subset H(\GL_4(F)\times \GL_2(F))=H(\GL_4(F)\times \{1_2\}).
\end{align*}
The classification of the lower $2 \times 2$ block of a matrix in $\GL_4(F)$ gives the following decomposition
\begin{align*}
    G^2_{\mathrm{open},1}=&\left(H\left\{ \left(\begin{smallmatrix}
        \left(\begin{smallmatrix}
        a_{11} & a_{12}\\
        a_{21} & a_{22}
    \end{smallmatrix}\right) & 0\\
    0 & 1_2
    \end{smallmatrix}\right) \colon a_{22} \; \mathrm{invertible}\right\} 
    \cup H\left\{ \left(\begin{smallmatrix}
        \left(\begin{smallmatrix}
        a_{11} & a_{12}\\
        a_{21} & a_{22}
    \end{smallmatrix}\right) & 0\\
    0 & 1_2
    \end{smallmatrix}\right)
     \colon a_{22}\neq0 \; \mathrm{diagonalizable}\right\}\right)\\&
    \sqcup H\left\{ \left(\begin{smallmatrix}
        \left(\begin{smallmatrix}
        a_{11} & a_{12}\\
        a_{21} & a_{22}
    \end{smallmatrix}\right) & 0\\
    0 & 1_2
    \end{smallmatrix}\right)
     \colon a_{22}\neq0 \; \mathrm{nilpotent}\right\},
\end{align*}
where each block is a $2\times 2$ matrix.

Note that
\begin{align*}
    H\left\{ \left(\begin{smallmatrix}
        \left(\begin{smallmatrix}
        a_{11} & a_{12}\\
        a_{21} & a_{22}
    \end{smallmatrix}\right) & 0\\
    0 & 1_2
    \end{smallmatrix}\right) \colon a_{22} \; \mathrm{invertible}\right\}=HM_2,
\end{align*}
where
$M_2=\GL_2(F)\times\GL_2(F)\times\{1_2\}\subset G$. 
And
\begin{align*}
    H\left\{ \left(\begin{smallmatrix}
        \left(\begin{smallmatrix}
        a_{11} & a_{12}\\
        a_{21} & a_{22}
    \end{smallmatrix}\right) & 0\\
    0 & 1_2
    \end{smallmatrix}\right) \colon a_{22}\neq0 \; \mathrm{diagonalizable}\right\}=HM_3,
\end{align*}
where $M_3=\GL_3(F)\times F^{\times}\times\{1_2\}\subset G$. 
Also notice that
\begin{align*}
    H\left\{ \left(\begin{smallmatrix}
        \left(\begin{smallmatrix}
        a_{11} & a_{12}\\
        a_{21} & a_{22}
    \end{smallmatrix}\right) & 0\\
    0 & 1_2
    \end{smallmatrix}\right) \colon a_{22}\neq0 \; \mathrm{nilpotent}\right\}=HZ_4,
\end{align*}
where
\begin{align*}
    Z_4=\left\{  \begin{pmatrix}
        a_{11} & a_{12} & 0 \\a_{21} & a_{22} & 0 \\ 0 & 0 & y
    \end{pmatrix}   \in \GL_4(F)\times\GL_2(F)\subset G\colon (y^{-1}a_{22})^2=0, a_{22}\neq 0 \right\}.
\end{align*}
Then we obtain the following expression for $G^2_{\mathrm{open},1}$:
\begin{align}\label{G2op1-1}
    G^2_{\mathrm{open},1}=(HM_2 \cup HM_3) \sqcup HZ_4,
\end{align}
where $HZ_4$ is clearly a closed subset in $G^2_{\mathrm{open},1}$. This forms the main steps to prove that  $T|_{G^2_{\mathrm{open},1}}=0$ for any $T \in \D^{\prime}(G)_{\widetilde{\chi}}$, which is $\mathfrak{z}$-finite as in Theorem \ref{prop 3.2}.

The idea to prove that $T|_{G^2_{\mathrm{open},1}}=0$ can be outlined as follows. We will prove in Section \ref{section 4.1} that 
$\D^{\prime}_{\widetilde{\chi}}(HM_2)=0$ and in Section \ref{section 4.2} that 
$\D^{\prime}_{\widetilde{\chi}}(HM_3)=0$. 
Applying Lemma \ref{lem 2.1} to the case $HM_2 \cup HM_3$, where $HM_2$ is an open $\widetilde{H}$-invariant submanifold, we obtain that 
\begin{align*}
    \mathcal{D}^{\prime}_{\widetilde{\chi}}(HM_2 \cup HM_3)=0.
\end{align*}
This implies that the distribution $T|_{G^2_{\mathrm{open},1}}$ must vanish on $HM_2 \cup HM_3$.
On the other hand, by Theorem \ref{thm 2.15}, the $\mathfrak{z}$-finiteness implies that the wavefront set of $T$ has the property that $\wf_x(T)\subset\mathcal{N}$, for all $x\in G$, 
where $\mathcal{N}$ is the nilpotent cone.
By Proposition \ref{prop 2.14}, the wavefront set is local and we have
\begin{align}\label{WFG21}
    \wf_x(T|_{G^2_{\mathrm{open},1}})\subset\mathcal{N}, \; \mathrm{for} \; \mathrm{all} \; x\in G^2_{\mathrm{open},1},
\end{align}
since $G^2_{\mathrm{open},1}$ is open in $G$.
In Section \ref{section 4.4}, we will prove that there is no non-zero distribution satisfying the above two conditions. We hence conclude that 
$T |_{G^2_{\mathrm{open},1}}=0.$

\subsubsection{The case of $M_2$} \label{section 4.1}
We prove here that 
\begin{align}\label{HM2}
    \mathcal{D}^{\prime}_{\widetilde{\chi}}(HM_2)=0.
\end{align}
Define 
\begin{align*}
    H_2:=\{ (x,x^{-\tau}) \colon x\in \GL_2^{\Delta}(F)   \}\subset  H=S\times S.
\end{align*}
and the extended group 
\begin{align*}
    \widetilde{H}_2:=\{1,\tau\}\ltimes H_2 \subset \widetilde{H},
\end{align*}
Denote the restriction of $\widetilde{\chi}$ to $\widetilde{H}_2$ by $\widetilde{\chi}_2$. 
After the identification of $H_2$ with $\GL_2(F)$, the semidirect product $\widetilde{H}_2:=\{1,\tau\}\ltimes H_2$ is given with the action $\tau(g)=g^{-\tau}$, 
and the character $\widetilde{\chi}_2$ is given by 
\begin{align*}
    \widetilde{\chi}_2|_{\GL_2(F)}=1 \quad\mathrm{and}\quad \widetilde{\chi}_2(\tau)=-1.
\end{align*}
Recall from Section \ref{ssec-SG} that 
\begin{align*}
    M_2=\GL_2(F)\times\GL_2(F)\times\{1_2\}\subset G.
\end{align*}
We may identify $M_2$ with $\GL_2(F) \times \GL_2(F)$. The action of $\widetilde{H}_2$ on $M_2$ is given by 
\begin{align*}
    g(x,y)=(gxg^{-1},gyg^{-1}),\quad g\in\GL_2(F),\quad {\rm and}\quad \tau(x,y)=(x^{\tau},y^{\tau}).
\end{align*}
The assertion in \eqref{HM2} follows from Lemma \ref{lem 2.2} and the following proposition.

\begin{prop}
With the notation as given above, the following holds 
\begin{align*}
    \mathcal{D}^{'}_{\widetilde{\chi}_2}(M_2)=0.
\end{align*}
\end{prop}

\begin{proof}
    Using Proposition 4.5 in \cite{MR2637582} we know such a distribution is invariant under the involution $(x,y)\mapsto(\overline{x},\overline{y})$, where
    $\overline{x}=\Tr(x)-x$ for $x\in\GL_2(F)$. 
Write $\omega=\begin{pmatrix}
        0 & 1 \\ -1 & 0
    \end{pmatrix}$. It is clear that  $\omega \overline{x} \omega ^{-1}=x^{\tau}$. 
It is an easy computation that such a distribution $T$ is also invariant under the action $\tau$, which implies that $T=0$.
\end{proof}

\subsubsection{The case of $M_3$}\label{section 4.2}

We prove here that 
\begin{align}\label{HM3}
    \mathcal{D}^{'}_{\widetilde{\chi}}(HM_3)=0.
\end{align}
We define a subgroup $H_3 \subset H$ that consists of elements of the form
\begin{align*}
    \left( \left(\begin{smallmatrix}
        a & 0 & * &0 & 0 & 0 \\ 0 & 1 & * & 0 & 0 & 0 \\ 0 & 0 & a & 0 & 0 & 0 \\ 0 & 0 & 0 & 1 & 0 & 0 \\0 & 0 & 0 &0 &a& 0\\ 0& 0 &0 &0 & 0 &1
    \end{smallmatrix}\right) ,\left(\begin{smallmatrix}
        a^{-1} & 0 & * &0 & 0 & 0 \\ 0 & 1 & * & 0 & 0 & 0 \\ 0 & 0 & a^{-1} & 0 & 0 & 0 \\ 0 & 0 & 0 & 1 & 0 & 0 \\0 & 0 & 0 &0 &a^{-1}& 0\\ 0& 0 &0 &0 & 0 &1
    \end{smallmatrix}  \right)              \right)
\end{align*}
and the extended group 
\begin{align*}
    \widetilde{H}_3:=\{1,\tau\}\ltimes H_3 \subset \widetilde{H},
\end{align*}
Denote the restriction of $\widetilde{\chi}$ to $\widetilde{H}_3$ by $\widetilde{\chi}_3$.
Recall from Section \ref{ssec-SG} that 
\begin{align*}
    M_3=\GL_3(F)\times F^{\times}\times\{1_2\}\subset G.
\end{align*}
To understand the action of $\widetilde{H}_3$ on $M_3$, we re-write the groups $H_3$ and $\widetilde{H}_3$ as follows. 

Write
\begin{align*}
    L_3=\left\{ \begin{pmatrix}
        a & 0 & 0 \\ 0 & 1 & 0 \\ 0 & 0 & a 
    \end{pmatrix}      \;|\;a\in F^{\times}  \right\}\quad {\rm and}\quad N_3=\left\{ \begin{pmatrix}
        1 & 0 & d \\ 0 & 1 & c \\ 0 & 0 & 1 
    \end{pmatrix}      \;|\;c,d\in F\right\}.
\end{align*}
We obtain the isomorphism
\begin{align*}
    H_3\cong L_3\ltimes(N_3\times N_3),
\end{align*}
where the semidirect product is defined by the action
\begin{align*}
    l(g_1,g_2)=(lg_1l^{-1},l^{-1}g_2l).
\end{align*}
Under this identification, the semidirect product $\widetilde{H}_3=\{1,\tau\}\ltimes H_3$ is given by the action
\begin{align*}
    \tau(l,g_1,g_2)=(l^{-1},g_2,g_1), 
\end{align*}
and the character 
$\widetilde{\chi}_3$ is given by
\begin{align*}
    \widetilde{\chi}_3(l,g_1,g_2)=\chi_{N_3}(g_1)\chi_{N_3}(g_2),\quad(l,g_1,g_2)\in H_3, \quad {\rm and}\quad \widetilde{\chi}_3(\tau)=-1,
\end{align*}
where
\begin{align*}
    \chi_{N_3} \begin{pmatrix}
        1 & 0 & d \\ 0 & 1 & c \\ 0 & 0 & 1 
    \end{pmatrix}:=\psi_{F}(d).
\end{align*}
After the identification of $M_3$ with  $\GL_3(F)\times F^{\times}$, the action of $\widetilde{H}_3$ on $M_3$ is given by 
\begin{align*}
        (l,g_1,g_2)(x,y)=(lg_1xg_2^{\tau}l^{-1},y)\quad {\rm and}\quad \tau(x,y)=(x^{\tau},y).
    \end{align*}
Now the assertion in \eqref{HM3} follows from Lemma \ref{lem 2.2} and the following proposition.

\begin{prop}\label{prop:M3}
With the notations as above, the following holds
\begin{align*}
    \mathcal{D}^{'}_{\widetilde{\chi}_3}(M_3)=0.
\end{align*}
\end{prop}
Since $M_3\cong\GL_3(F)\times F^{\times}$ and 
the action of $\widetilde{H}_3$ on $F^{\times}$ is trivial, to prove that $\mathcal{D}^{'}_{\widetilde{\chi}_3}(M_3)=0$ in Proposition~\ref{prop:M3}, it is enough to prove that 
\begin{align}\label{GL3}
    \mathcal{D}^{'}_{\widetilde{\chi}_3}(\GL_3(F))=0
\end{align}
due to Lemma \ref{lem-prod}.

In order to prove the assertion in \eqref{GL3}, we consider the following filtration of $\widetilde{H}_3$-stable closed submanifolds of $\GL_3(F)$:
\begin{align}\label{GL3filter}
    Z_{3,0}:=\GL_3(F) \supset Z_{3,1}  \supset Z_{3,2} \supset Z_{3,3} \supset Z_{3,4}:=\emptyset, 
\end{align}
where  
\begin{align*}
    Z_{3,1}= \left\{ \left(\begin{smallmatrix}
        * & * & *\\
        * & * & *\\
        * & * & 0
    \end{smallmatrix}\right) \in  \GL_3(F)   \right\}, \ 
    Z_{3,2}= \left\{ \left(\begin{smallmatrix}
        * & * & a\\
        * & * & *\\
        a & * & 0
    \end{smallmatrix}\right) \in  \GL_3(F)    \right\}, \ {\rm and}\ Z_{3,3}= \left\{ \left(\begin{smallmatrix}
        * & * & 0\\
        * & * & *\\
        0 & * & 0
    \end{smallmatrix}\right) \in  \GL_3(F)     \right\}.
\end{align*}
According to Lemma \ref{lem 2.1}, in order to prove the assertion in \eqref{GL3}, it suffices to prove that 
\begin{align}\label{GL3-i}
    \mathcal{D}^{'}_{\widetilde{\chi}_3}(Z_{3,i}\setminus Z_{3,i+1})=0
\end{align}
for each $0 \leq i \leq 3$, which will be proved in the following lemmas.

First we show the case of $i=0$:
\begin{align}\label{GL3-0}
    \mathcal{D}^{'}_{\widetilde{\chi}_3}(Z_{3,0}\setminus Z_{3,1})=\mathcal{D}^{'}_{\widetilde{\chi}_3}(\GL_3(F)\setminus Z_{3,1})=0.
\end{align}
Since $\GL_2(F) \times F^{\times}$ is an $\widetilde{H}_3$ slice of $\GL_3(F)\setminus Z_{3,1}$, by Lemma \ref{lem 2.2}, the assertion in \eqref{GL3-0} follows from 
the following lemma.

\begin{lem}\label{lem 4.2}
Let $\widetilde{H}_{3,1}=\{1,\tau\} \ltimes L_3=\{1,\tau\}\ltimes F^{\times}$, and $\widetilde{\chi}_{3,1}$ be the restriction of $\widetilde{\chi}_3$ to $\widetilde{H}_{3,1}$. Consider the action of $\widetilde{H}_{3,1}$ on $\GL_2(F)\times F^{\times} \hookrightarrow\GL_3(F)$ induced by the action of $\widetilde{H}_3$ on $\GL_3(F)$, we have
\begin{align*}
\D^{\prime}_{\widetilde{\chi }_{3,1}    }\left(  \GL_2(F)\times F^{\times}    \right)=0.
\end{align*}
\end{lem}

\begin{proof}
Since $\widetilde{H}_{3,1}$ acts trivially on $F^{\times}$, it suffices to prove that 
\begin{align}\label{GL2-0}
    \D^{\prime}_{\widetilde{\chi }_{3,1}    }(  \GL_2(F)    )=0,
\end{align}
according to Lemma \ref{lem-prod}.

Extend the action on $\GL_2(F)$ to that on $\gl_2(F)$ and consider the determinant map
\begin{align*}
    \mathrm{det}: \; \gl_2(F) \rightarrow F.
\end{align*}
It is clear the action of $\widetilde{H}_{3,1}$ preserves each fiber of the determinant map. Since $\GL_2(F)= \mathrm{det}^{-1}(F^{\times})$ is a union of fibers, according to Corollary \ref{cor 2.10}, the assertion in \eqref{GL2-0} follows from the following assertion:
\begin{align}\label{GL2-1}
    \D^{\prime}_{\widetilde{\chi }_{3,1}    }(  \gl_2(F)    )=0.
\end{align}
To prove the assertion in \eqref{GL2-1} we identify $\gl_2(F)$ with $(F\times F)\times(F\times F)$ via the map
\begin{align*}
    \begin{pmatrix}
        x_{11} & x_{12} \\
        x_{21} & x_{22} 
    \end{pmatrix} \mapsto ((x_{11},x_{22}),(x_{12},x_{21})).
\end{align*}
Then $\widetilde{H}_{3,1}$ acts on the first $F \times F$ trivially and acts on the second $F\times F$ via
\begin{align*}
a(x_{12},x_{21})=(ax_{12},a^{-1}x_{21}), a\in F^{\times},\quad \mathrm{and} \quad \tau(x_{12},x_{21})=(x_{21},x_{12}).
\end{align*}
It suffices to show $\D^{\prime}_{\widetilde{\chi}_{3,1}}(F\times F)=0$ for the action given above, which is exactly \cite[Lemma 4.6]{MR2637582}. We are done. 
\end{proof}

The case of $i=1$ is given by the following lemma. 
\begin{lem}
    $\mathcal{D}^{'}_{\widetilde{\chi}_3}(Z_{3,1} \setminus Z_{3,2})=0$.
\end{lem}
\begin{proof}
    Following from the calculation in \cite[Lemma 8.4]{MR2763736}, we take
    \begin{align*}
        x=\begin{pmatrix}
            x_{11} & x_{22} & x_{13} \\ x_{21} & x_{22} & x_{23} \\ x_{31} &x_{32} & 0
        \end{pmatrix}\in Z_{3,1}\setminus Z_{3,2}
    \end{align*}
    and write 
    \begin{align*}
        u(x,t)=\begin{pmatrix}
            1 & 0 & x_{13}t\\ 0 & 1 & x_{23}t \\ 0 & 0 &  1
        \end{pmatrix} \quad\mathrm{and}\quad v(x,t)=\begin{pmatrix}
            1 & 0 & 0 \\ 0 & 1 & 0 \\ tx_{31} & tx_{32} & 1 
        \end{pmatrix}.
    \end{align*}
    Then $u(x,t)x=xv(x,t)$, and $\chi_3(u(x,t),v(x,t)^{-\tau})=\psi(x_{13}t-x_{31}t)\neq 1$ for a suitable chosen $t\in F$. Then the lemma follows from Theorem \ref{thm 2.1}
    and Remark \ref{rem .2}.
\end{proof}

Here is the case of $i=2$.
\begin{lem}\label{i=2}
    $\mathcal{D}^{'}_{\widetilde{\chi}_3}(Z_{3,2} \setminus Z_{3,3})=0$.
\end{lem}
\begin{proof}
    The submanifold $Z_{3,2} \setminus Z_{3,3}$ has an $H_3$ slice given by
\begin{align*}
    \begin{pmatrix}
        0 & 0 & a\\
        0 & x & z\\
        a & y & 0
    \end{pmatrix}, \; a,x \in F^{\times}, \; y,z \in F,
\end{align*}
which is stable under the action of $\widetilde{H}_{3,1}$.
We identify this slice with $(F^{\times}\times F^{\times})\times(F\times F)$ and identify $\widetilde{H}_{3,1}$ with $\{1,\tau\}\ltimes F^{\times}$. Then $\widetilde{H}_{3,1}$ acts trivially on $(F^{\times}\times F^{\times})$ and acts on $(F\times F)$ by
\begin{align*}
    t(y,z)=(ty,t^{-1}z), \; \mathrm{and} \; \tau(y,z)=(z,y).
\end{align*}
Again this lemma follows from Lemma \ref{lem 2.2} and \cite[Lemma 4.6]{MR2637582}. 
\end{proof}
Finally, we consider the case that $i=3$. 
\begin{lem}
    $\mathcal{D}^{'}_{\widetilde{\chi}_3}(Z_{3,3})=0$.
\end{lem}
\begin{proof}
    The submanifold $Z_{3,3}$ has an $H_3$ slice given by
\begin{align*}
    \begin{pmatrix}
        x & 0 & 0\\
        0 & 0 & z\\
        0 & y & 0
    \end{pmatrix}, \; x,y,z \in F^{\times},
\end{align*}
which is stable under the action of $\widetilde{H}_{3,1}$.
We identify the slice with $F^{\times}\times(F^{\times}\times F^{\times})$. It is clear that $\widetilde{H}_{3,1}$ acts trivially on the first $x\in F^{\times}$, and acts on $(F^{\times}\times F^{\times})$ by
\begin{align*}
    t(y,z)=(ty,t^{-1}z), \; \mathrm{and} \; \tau(y,z)=(z,y).
\end{align*}
As before, the assertion in the lemma follows from Lemma \ref{lem-prod}, Lemma \ref{lem 2.2} and the assertion that 
\[
\mathcal{D}^{'}_{\widetilde{\chi}_{3,1}}(F^{\times}\times F^{\times})=0.
\]
To show that $\mathcal{D}^{'}_{\widetilde{\chi}_{3,1}}(F^{\times}\times F^{\times})=0$, we note the action of $\widetilde{H}_{3,1}$ on $(F^{\times}\times F^{\times})$ is the restriction of action of $\widetilde{H}_{3,1}$ on $F\times F$ as in the proof of Lemma \ref{i=2}. Consider the multiplication map 
\begin{align*}
    F\times F \rightarrow F, \; (y,z) \rightarrow yz.
\end{align*}
It is clear that the action of $\widetilde{H}_{3,1}$ preserves each fiber. Since $F^{\times}\times F^{\times}$ is a union of fibers, by the same argument, the assertion in the lemma 
follows from Corollary \ref{cor 2.10} and  \cite[Lemma 4.6]{MR2637582}. 
\end{proof}

\subsubsection{The case of $Z_4$}\label{section 4.4}

Recall that 
\begin{align*}
    Z_4=\left\{  \begin{pmatrix}
        a_{11} & a_{12} & 0 \\a_{21} & a_{22} & 0 \\ 0 & 0 & y
    \end{pmatrix}   \in \GL_4(F)\times\GL_2(F)\subset G \colon (y^{-1}a_{22})^2=0, a_{22}\neq 0 \right\}.
\end{align*}
From the discussion in Section \ref{section 3} and the results in Sections \ref{section 4.1}, \ref{section 4.2}, we know that $T|_{G^2_{\mathrm{open},1}}$ is supported in $HZ_4$. 
From the general theory (Theorem \ref{thm 2.15}), the wavefront set of $T$ is of nilpotent type \eqref{WFG21}. Hence it is enough to prove the following more general proposition, 
in order to prove that $T|_{G^2_{\mathrm{open},1}}=0$.

\begin{prop}\label{prop 4.4}
    Let $T^{\prime} \in \mathcal{D}^{\prime}_{\chi}(G^2_{\mathrm{open},1})$ be a distribution such that
    \begin{align*}
        \wf_x(T^{\prime})\subset\mathcal{N}, \; \mathrm{for} \; \mathrm{all} \; x\in G^2_{\mathrm{open},1}.
    \end{align*}
    If $T^{\prime} |_{HM_2 \cup HM_3} =0 $, i.e., $T^{\prime}$ is supported in $HZ_4$, then $T^{\prime} =0$.
\end{prop}

The proof of Proposition \ref{prop 4.4} takes two steps. The first step is to understand the structure of the support of the distribution $T'$ as given in Proposition \ref{prop 4.4}, if it is not zero, 
and the second step is to apply the wavefront set property to the deduced structure of the possible support of $T'$.

In order to understand the structure of the possible support of $T'$, write $\mathfrak{Z}_4$ to be the set of all matrices of the form
\begin{align*}
    x=\left(\begin{smallmatrix}
        x_{11} & x_{12} & x_{13} & 0 & 0 & 0 \\x_{21} & x_{22} & x_{23} & 0 & 0 & 0 \\ x_{31} & x_{32} & 0 & 0 & 0 & 0 \\ 0 & 0 & 0 & 1 & 0 & 0\\ 0& 0 & 0 & 0 & 0 & 1 \\ 0& 0 & 0 & 0 & 1 & 0
    \end{smallmatrix}\right).
\end{align*}
Define  
\begin{align*}
\mathfrak{Z}_{4,1}=\{x\in\mathfrak{Z}_{4}:\; x_{13}=x_{31}\}\quad {\rm and }\quad 
    \mathfrak{Z}_{4,2}=\{x\in\mathfrak{Z}_{4,1}:\; x_{13}=x_{31}=0\}.
\end{align*}
It is clear that $\mathfrak{Z}_{4,2}\subset \mathfrak{Z}_{4,1}\subset \mathfrak{Z}_4$. 
We also denote by $\mathfrak{Z}_{4,2}^{'}$ all matrices in $\mathfrak{Z}_{4,2}$ of the form
\begin{align*}
    x=\left(\begin{smallmatrix}
        x_{11} & 0 & 0 & 0 & 0 & 0 \\0 & 0 & x_{23} & 0 & 0 & 0 \\ 0 & x_{32} & 0 & 0 & 0 & 0 \\ 0 & 0 & 0 & 1 & 0 & 0\\ 0& 0 & 0 & 0 & 0 & 1 \\ 0& 0 & 0 & 0 & 1 & 0
    \end{smallmatrix}\right).
\end{align*}
Write
\begin{align*}
\mathfrak{Z}_{4,3}^{1\;'}=\{ x\in\mathfrak{Z}_{4,2}^{'}:\;x_{23}=x_{32}\}\quad {\rm and}\quad \mathfrak{Z}_{4,3}^{2\;'}=\{ x\in\mathfrak{Z}_{4,2}^{'}:\;x_{23}x_{32}+x_{11}=0\}.
\end{align*}
We consider the following filtration on $Z_4$:
\begin{align*}
    HZ_4\supset H\mathfrak{Z}_{4,1}\supset H\mathfrak{Z}_{4,2}\supset \left(H\mathfrak{Z}_{4,3}^{1\;'}\cup H \mathfrak{Z}_{4,3}^{2\;'}\right)\supset H\mathfrak{Z}_{4,3}^{2\;'}\supset\emptyset.
\end{align*}

We first show that $\mathcal{D}^{'}_{\chi}(HZ_4\setminus H\mathfrak{Z}_{4,1})=0$, which implies that $T'$ is supported in $H\mathfrak{Z}_{4,1}$. We next prove that $\mathcal{D}^{'}_{\widetilde{\chi}}(H\mathfrak{Z}_{4,1}\setminus H\mathfrak{Z}_{4,2})=0$, which implies that $T'$ is supported in $H\mathfrak{Z}_{4,2}$. Finally we confirm that $T'$ is supported in $H\mathfrak{Z}_{4,3}^{2,'}$ by showing  $\mathcal{D}^{'}_{\chi}\left(H\mathfrak{Z}_{4,2}\setminus \left(  H\mathfrak{Z}_{4,3}^{1\;'}\cup H\mathfrak{Z}_{4,3}^{2\;'}  \right)\right)=0$ and $\mathcal{D}_{\widetilde{\chi}}^{'}(H\mathfrak{Z}_{4,3}^{1\;'})=0$. Hence we obtain that any distribution $T'$ as given in Proposition \ref{prop 4.4} is possibly supported on $H\mathfrak{Z}_{4,3}^{2,'}$. 

\begin{lem}\label{Z41}
    $\mathcal{D}^{'}_{\chi}(HZ_4\setminus H\mathfrak{Z}_{4,1})=0$.
\end{lem}

\begin{proof}
    Following from the calculation in \cite[Lemma 5.4]{MR2763736}, for any 
    \begin{align*}
        x=\left(\begin{smallmatrix}
            x_{11} & x_{12} & x_{13} & 0 & 0 & 0 \\ x_{21} & x_{22} & x_{23} & 0 & 0 & 0 \\ x_{31}& x_{32}  & 0 & 0 & 0 & 0 \\ 0 & 0 & 0 & 1 & 0 & 0 \\ 0 &0& 0 & 0 & 0 & 1 \\ 0 & 0 & 0 & 0 & 1 & 0
        \end{smallmatrix}\right)\in \FZ_{4}^{\prime}\setminus \FZ_{4,1}^{\prime},
    \end{align*}
     write 
    \begin{align*}
        u(x,t)=\left(\begin{smallmatrix}
            1 & 0 & x_{13}t & 0 & 0 & 0 \\ 0 & 1 & x_{23}t & 0 & 0 & 0 \\ 0 & 0  & 1 & 0 & 0 & 0 \\ 0 & 0 & 0 & 1 & 0 & 0 \\ 0 & 0 & 0 & 0 & 1 & 0 \\ 0 & 0 & 0 & 0 & 0 & 1
        \end{smallmatrix}\right)\quad\mathrm{and}\quad v(x,t)=\left(\begin{smallmatrix}
            1 & 0 & 0 & 0 & 0 & 0 \\ 0 & 1 & 0 & 0 & 0 & 0 \\ tx_{31} & tx_{32}  & 1 & 0 & 0 & 0 \\ 0 & 0 & 0 & 1 & 0 & 0 \\ 0 & 0 & 0 & 0 & 1 & 0 \\ 0 & 0 & 0 & 0 & 0 & 1
        \end{smallmatrix}\right).
    \end{align*}
    Then $u(x,t)x=xv(x,t)$. Since $x_{13}\neq x_{31}$, we have
    \begin{align*}
        \chi(u(x,t),v(x,t)^{-\tau})=\psi_F(x_{13}t-x_{31}t)\neq 1
    \end{align*}
    for a suitable chosen $t\in F$. Thus the lemma follows from Theorem \ref{thm 2.1} and Remark \ref{rem .2} as before.
\end{proof}

\begin{lem}\label{Z42}
    $\mathcal{D}^{'}_{\widetilde{\chi}}(H\mathfrak{Z}_{4,1}\setminus H\mathfrak{Z}_{4,2})=0$.
\end{lem}

\begin{proof}
    According to Lemma 5.5 in \cite{MR2763736}, we have that $H\mathfrak{Z}_{4,1}\setminus H\mathfrak{Z}_{4,2}=H\mathfrak{Z}_{4,1}^{'}$, where $\mathfrak{Z}_{4,1}^{'}$ consists all matrices of the form
\begin{align*}
    x=\left(\begin{smallmatrix}
        0 & 0 & a & 0 & 0 & 0 \\0 & x_{22} & 0 & 0 & 0 & 0 \\ a & 0 & 0 & 0 & 0 & 0 \\ 0 & 0 & 0 & 1 & 0 & 0\\ 0& 0 & 0 & 0 & 0 & 1 \\ 0& 0 & 0 & 0 & 1 & b
    \end{smallmatrix}\right).
\end{align*}
Assume that there is some non-zero $(H,\chi)$-equivariant distribution on it. Since this matrix is symmetric, we can take $h$ to be the identity in Lemma \ref{lem 2.7}. Then each $H$ orbit in $H\mathfrak{Z}_{4,1}^{'}$ is $\tau$-stable and any $(H,\chi)$ invariant distribution on this $H$ orbit will be $\tau$-invariant. In other words, every $(\widetilde{H},\widetilde{\chi})$-equivariant distribution on each $H$ orbit will vanish. By Remark \ref{rem 2.3}, we must have that $\mathcal{D}^{'}_{\widetilde{\chi}}(H\mathfrak{Z}_{4,1}^{\prime})=0$.
\end{proof}

\begin{lem}\label{Z43}
$\mathcal{D}^{'}_{\chi}\left(H\mathfrak{Z}_{4,2}\setminus \left(  H\mathfrak{Z}_{4,3}^{1\;'}\cup H\mathfrak{Z}_{4,3}^{2\;'}  \right)\right)=0$.
\end{lem}

\begin{proof}
    Following the calculation in \cite[Lemma 5.6]{MR2763736}, let $x\in H\mathfrak{Z}_{4,2} \setminus \left(H\mathfrak{Z}_{4,3}^{1\;'}\cup H\mathfrak{Z}_{4,3}^{2\;'}\right) $ be of the form
 \begin{align*}
        x=\left(\begin{smallmatrix}
            x_{11} & 0 & 0 & 0 & 0 & 0 \\ 0 & 0 & x_{23} & 0 & 0 & 0 \\ 0 & x_{32}  & 0 & 0 & 0 & 0 \\ 0 & 0 & 0 & 1 & 0 & 0 \\ 0 &0& 0 & 0 & 0 & 1 \\ 0 & 0 & 0 & 0 & 1 & 0
        \end{smallmatrix}\right).
    \end{align*}
    Write 
    \begin{align*}
        u(x,t)=\left(\begin{smallmatrix}
            1 & 0 & x_{11}x_{32}^{-1}t & 0 & 0 & 0 \\ t & 1 & 0 & x_{23}t & 0 & 0 \\ 0 & 0  & 1 & 0 & 0 & 0 \\ 0 & 0 & t & 1 & 0 & 0 \\ 0 & 0 & 0 & 0 & 1 & 0 \\ 0 & 0 & 0 & 0 & t & 1
        \end{smallmatrix}\right)\quad\mathrm{and}\quad v(x,t)=\left(\begin{smallmatrix}
            1 & t & 0 & 0 & 0 & 0 \\ 0 & 1 & 0 & 0 & 0 & 0 \\ x_{11}x_{23}^{-1}t & 0  & 1 & t & 0 & 0 \\ 0 & x_{32}t & 0 & 1 & 0 & 0 \\ 0 & 0 & 0 & 0 & 1 & t \\ 0 & 0 & 0 & 0 & 0 & 1
        \end{smallmatrix}\right).
    \end{align*}
    Then $u(x,t)x=xv(x,t)$ and we have
    \begin{align*}
        \chi(u(x,t),v(x,t)^{-\tau})=\psi_F((x_{32}^{-1}-x_{23}^{-1})(x_{11}+x_{23}x_{32})t)\neq 1
    \end{align*}
    for a suitable chosen $t\in F$. Thus the lemma follows from Theorem \ref{thm 2.1} and Remark \ref{rem .2} as before.
\end{proof}

\begin{lem}\label{Z44}
    $\mathcal{D}^{'}_{\widetilde{\chi}}(H\mathfrak{Z}_{4,3}^{1\;'})=0$.
\end{lem}

\begin{proof}
We use the same argument as in the proof of Lemma \ref{Z42}. 
    Assume taht there is some non-zero $(H,\chi)$-equivariant distribution on $H\mathfrak{Z}_{4,3}^{1\;'}$. 
    Since any matrix in $H\mathfrak{Z}_{4,3}^{1\;'}$ is symmetric, we can take $h$ to be the identity in Lemma \ref{lem 2.7}. Then each $H$ orbit in $H\mathfrak{Z}_{4,3}^{1\;'}$ is $\tau$-stable and any $(H,\chi)$ invariant distribution on this $H$ orbit will be $\tau$-invariant. In other words, every $(\widetilde{H},\widetilde{\chi})$-equivariant distribution on each $H$ orbit will vanish. By Remark \ref{rem 2.3}, we must have that $\mathcal{D}^{'}_{\widetilde{\chi}}(H\mathfrak{Z}_{4,3}^{1\;'})=0$. 
\end{proof}

From the first step as discussion given above, the distribution $T'$ as given in Proposition \ref{prop 4.4} is possibly supported on $H\mathfrak{Z}_{4,3}^{2\;'}$. The second step to finish the proof of 
Proposition \ref{prop 4.4} is to apply the wavefront set condition to the distribution $T'$ on $H\mathfrak{Z}_{4,3}^{2\;'}$ and reach a contradiction, which confirms that $T'$ must be zero and completes the proof of Proposition \ref{prop 4.4}.

If $T^{\prime} \neq0 $, since $T^{\prime}$ is $H$-equivariant, there exists $x \in \mathfrak{Z}_{4,3}^{2\;'}$  such that  $x \in \supp(T^{\prime})$. From the definition of the wavefront set, we have that 
$(x,0) \in \wf_x(T^{\prime}).$ 
By Theorem \ref{thm 2.12}, this implies the conormal bundles have the following property:
\begin{align*}
\cn^{G^4_{\mathrm{open}}}_{H\mathfrak{Z}_{4,3}^{2\;'},x}=\cn^G_{H\mathfrak{Z}_{4,3}^{2\;'},x} \subset \wf_x(T').
\end{align*}
On the other hand, according to Lemma 5.8 in \cite{MR2763736}, if we take
\begin{align*}
    x^{'}=\left(\begin{smallmatrix}
        1 & 0 & 0 & 0 & 0 & 0 \\0 & 0 & x_{23} & 0 & 0 & 0 \\ 0 & x_{32} & 0 & 0 & 0 & 0 \\ 0 & 0 & 0 & -x_{23}x_{32} & 0 & 0\\ 0& 0 & 0 & 0 & 0 & 0 \\ 0& 0 & 0 & 0 & 0 & 0
    \end{smallmatrix}\right),
\end{align*}
then 

\begin{align*}
    T_x(H\mathfrak{Z}_{4,3}^{2\;'})\subset (Fx^{\prime})^{\bot}
\end{align*}
where $\bot$ denotes the orthogonal complement with respect to the real trace form. This implies
\begin{align*}
    x^{\prime} \in \cn^G_{H\mathfrak{Z}_{4,3}^{2\;'},x} \subset \wf_x(T').
\end{align*}
Note under the identification of  
    $x^{-1}T_x(H\mathfrak{Z}_{4,3}^{2\;'})$ with a subspace of $T_e(G)$, $x^{\prime}$ is identified with $x^{\prime}x$ in $T_e(G)^*$.
However,
\begin{align*}
    x^{\prime}x=x_{23}x_{32}\left(\begin{smallmatrix}
        -1 & 0 & 0 & 0 & 0 & 0 \\0 & 1 & 0 & 0 & 0 & 0 \\ 0 & 0 & 1 & 0 & 0 & 0 \\ 0 & 0 & 0 & -1 & 0 & 0\\ 0& 0 & 0 & 0 & 0 & 0 \\ 0& 0 & 0 & 0 & 0 & 0
    \end{smallmatrix}\right)
\end{align*}
which is not nilpotent.
This contradicts to the assumption of Proposition \ref{prop 4.4} on the structure of the wavefront set of $T'$. 
Therefore we must have that $T^{\prime}=0$. We are done.

\subsection{The case of $G^2_{\mathrm{open},2}$}\label{ssec-G22}
In order to study the open subset $G^2_{\mathrm{open},2}$ of $G$, we introduce the following stratification:
\begin{align}\label{G2op2-1}
    G^2_{\mathrm{open},2}=(HM_2 \cup H\check{M}_3)\sqcup HW_4,
\end{align}
where $\check{M}_3=\{I_2\} \times \GL_3(F)\times F^{\times} \subset G$ and
\begin{align*}
    W_4=\left\{  \begin{pmatrix}
        y & 0 & 0 \\0 & a_{11} & a_{12} \\ 0 & a_{21} & a_{22}
    \end{pmatrix}   \in \GL_2(F)\times\GL_4(F)\subset G\colon (y^{-1}a_{22})^2=0, a_{22}\neq 0 \right\}.
\end{align*}
By a parallel argument as in Section \ref{section 4.2}, one can show that $\D^{\prime}_{\widetilde{\chi}}(H\check{M}_3)=0$, which implies that the distribution $T|_{G_{\mathrm{open},2}^2}$ must vanish on $HM_2\cup H\check{M}_3$. By applying the same arguments as in Section \ref{section 4.4} to the case $HW_4$, one obtains that $T|_{G_{\mathrm{open,2}}^2}=0$. Combining with Section \ref{ssec-G21}, we complete the proof of Proposition \ref{0=TG3}.

\section{Proof of Proposition \ref{0=TG4}}\label{sec-TG4}

We prove Proposition \ref{0=TG4}, namely, the following vanishing property
     \begin{align}\label{G43}
    \mathcal{D}^{\prime}_{\widetilde{\chi}}(G^4_{\mathrm{open}} \setminus G^3_{\mathrm{open}})=0
\end{align}
for any $T \in \D^{\prime}(G)_{\widetilde{\chi}}$, which is $\mathfrak{z}$-finite as in Theorem \ref{prop 3.2}. 
Recall from \eqref{G3op} and \eqref{G4op}, one has that 
\[
G^4_{\mathrm{open}} \setminus G^3_{\mathrm{open}}=G_R
\]
with the rank-matrix $R=\begin{pmatrix}
        3 & 2 \\2 & 1
    \end{pmatrix}$. For simplicity, we write $Z_6=G_R$. Hence we have to show that 
\begin{align}\label{0=Z6}
    \mathcal{D}^{'}_{\widetilde{\chi}}(Z_6)=0.
\end{align}
To prove the assertion in \eqref{0=Z6}, as in \cite[Section 6]{MR2763736}, we consider the following filtration of successively closed $H$-stable subsets of $Z_6$
\begin{align}\label{Z6filter}
    Z_6=Z_{6,0} \supset Z_{6,1} \supset Z_{6,2} \supset Z_{6,3} \supset  Z_{6,4} \supset Z_{6,5} \supset Z_{6,6}=\emptyset,
\end{align}
where each term in the filtration is defined as follows.

Denote by $\mathfrak{Z}_6$ all matrices in $ Z_6$ of the form
\begin{align*}
    x=\left(\begin{smallmatrix}
        * & * & * & * & x_{15} & 0 \\ * & * & * & * & x_{25} & 0 \\ * & * & 0 & 0 & x_{35} & 0 \\ * & * & 0 & 0 & x_{45} & 0 \\ x_{51} & x_{52} & x_{53} & x_{54} & 0 & 0 \\0 & 0 & 0 & 0 & 0 & 1
    \end{smallmatrix}\right)\in Z_6.
\end{align*}
Write 
\begin{align*}
    \mathfrak{Z}_{6,1}=\{ x\in\mathfrak{Z}_6\;:\;x_{35}=x_{53}   \}\quad {\rm and}\quad \mathfrak{Z}_{6,2}=\{ x\in\mathfrak{Z}_6\;:\;x_{35}=x_{53}=0   \}.
\end{align*}
They are both relatively $H$ stable closed submanifolds of $\mathfrak{Z}_6$ (Definition \ref{defn-1}). Hence both 
\begin{align*}
    Z_{6,1}:=H\mathfrak{Z}_{6,1} \quad\mathrm{and} \quad Z_{6,2}:=H\mathfrak{Z}_{6,2}
\end{align*}
are closed submanifolds of $Z_6$.
Denote by $\mathfrak{Z}_{6,2}^{'}$ all matrices in $\mathfrak{Z}_{6,2}$ of the form
\begin{align*}
    \left(\begin{smallmatrix}
        x_{11} & x_{12} & x_{13} & 0 & 0 & 0 \\ x_{21} & x_{22} & x_{23} & 0 & 0 & 0 \\ x_{31} & x_{32} & 0 & 0 & 0 & 0 \\ 0 & 0 & 0 & 0 & 1 & 0 \\ 0 & 0 & 0 & 1 & 0 & 0 \\0 & 0 & 0 & 0 & 0 & 1
    \end{smallmatrix}\right),
\end{align*}
which also forms an $H$ slice of $Z_{6,2}$. Set
\begin{align*}
    \mathfrak{Z}_{6,3}^{'}=\{ x\in\mathfrak{Z}_{6,2}^{'}\;:\;x_{13}=x_{31}  \}\quad {\rm and}\quad \mathfrak{Z}_{6,4}^{'}=\{ x\in\mathfrak{Z}_{6,2}^{'}\;:\;x_{13}=x_{31}=0 \}.
\end{align*}
They are both relatively $H$ stable closed submanifolds of $\mathfrak{Z}_{6,2}^{'}$ as well. Thus both
\begin{align*}
    Z_{6,3}:=H\mathfrak{Z}_{6,3}^{'} \;\mathrm{and}\;Z_{6,4}:=H\mathfrak{Z}_{6,4}^{'}
\end{align*}
are closed submanifolds of $Z_{6,2}$.
Finally, denote by $\mathfrak{Z}_{6,4}^{''}$ all matrices in $\mathfrak{Z}_{6,4}^{'}$ of the form
\begin{align*}
    \left(\begin{smallmatrix}
        x_{11} & 0 & 0 & 0 & 0 & 0 \\ 0 & 0 & x_{23} & 0 & 0 & 0 \\ 0 & x_{32} & 0 & 0 & 0 & 0 \\ 0 & 0 & 0 & 0 & 1 & 0 \\ 0 & 0 & 0 & 1 & 0 & 0 \\0 & 0 & 0 & 0 & 0 & 1
    \end{smallmatrix}\right),
\end{align*}
which also forms an $H$ slice of $Z_{6,4}$. Set
\begin{align*}
    \mathfrak{Z}_{6,5}^{''}=\{x\in\mathfrak{Z}_{6,4}^{''}\;:\;x_{23}=x_{32}   \}
\end{align*}
which is a relatively $H$ stable closed submanifold of $\mathfrak{Z}_{6,4}^{''}$. Hence the set 
\begin{align*}
    Z_{6,5}:=H\mathfrak{Z}_{6,5}^{''}
\end{align*}
is also a closed submanifold of $Z_{6,4}$.

To prove the assertion in \eqref{0=Z6}, by Lemma \ref{lem 2.1}, it is enough to show that  
\begin{align}\label{Z0-5}
    \mathcal{D}^{'}_{\widetilde{\chi}}(Z_{6,i}\setminus Z_{6,i+1})=0
\end{align}
with $i=0,1,2,\cdots,5$.

First, we prove the case of $i=0$. 

\begin{lem}
    $\mathcal{D}^{'}_{\chi}(Z_6\setminus Z_{6,1})=0$.
\end{lem}

\begin{proof}
    Following from the calculation in \cite[Lemma 6.2]{MR2763736}, for any 
    \begin{align*}
        x=\left(\begin{smallmatrix}
            * & * & * & * & x_{15} & 0 \\ * & * & * & * & x_{25} & 0 \\ *& *  & 0 & 0 & x_{35} & 0 \\ * & * & 0 & 0 & x_{45} & 0 \\ x_{51} & x_{52} & x_{53} & x_{54} & 0 & 0 \\ 0 & 0 & 0 & 0 & 0 & 1
        \end{smallmatrix}\right)\in \FZ_6\setminus \FZ_{6,1}
    \end{align*}
    and write 
    \begin{align*}
        u(x,t)=\left(\begin{smallmatrix}
            1 & 0 & 0 & 0 & x_{15}t & 0 \\ 0 & 1 & 0 & 0 & x_{25}t & 0 \\ 0 & 0  & 1 & 0 & x_{35}t & 0 \\ 0 & 0 & 0 & 1 & x_{45}t & 0 \\ 0 & 0 & 0 & 0 & 1 & 0 \\ 0 & 0 & 0 & 0 & 0 & 1
        \end{smallmatrix}\right)\quad\mathrm{and}\quad v(x,t)=\left(\begin{smallmatrix}
            1 & 0 & 0 & 0 & 0 & 0 \\ 0 & 1 & 0 & 0 & 0 & 0 \\ 0 & 0  & 1 & 0 & 0 & 0 \\ 0 & 0 & 0 & 1 & 0 & 0 \\ tx_{51} & tx_{52} & tx_{53} & tx_{54} & 1 & 0 \\ 0 & 0 & 0 & 0 & 0 & 1
        \end{smallmatrix}\right).
    \end{align*}
    Then $u(x,t)x=xv(x,t)$, and $\chi(u(x,t),v(x,t)^{-\tau})=\psi_F((x_{35}-x_{53})t)\neq 1$ for suitable chosen $t\in F$. Then the lemma follows from Theorem \ref{thm 2.1}
    and Remark \ref{rem .2}.
\end{proof}

The case of $i=1$ is given by 

\begin{lem}
    $\mathcal{D}^{'}_{\chi}(Z_{6,1}\setminus Z_{6,2})=0$.
\end{lem}

\begin{proof}
    By Lemma 6.3 in \cite{MR2763736}, $Z_{6,1}\setminus Z_{6,2}$ has the $H$ slice $\mathfrak{Z}$ given by
\begin{align*}
z =\left(\begin{smallmatrix}
    x_{11} & x_{12} & 0 & y_1 & 0 & 0\\
    x_{21} & x_{22} & 0 & y_2 & 0 & 0\\
    0 & 0 & 0 & 0 & a & 0\\
    w_1 & w_2 & 0 & 0 & 0 & 0\\
    0 & 0 & a & 0 & 0 & 0\\
    0 & 0 & 0 & 0 & 0 & 1
    \end{smallmatrix}\right).
\end{align*}
Write $\mathfrak{Z^{\prime}}= \{ z\in \mathfrak{Z} \; \mathrm{with}  \; w_2 =y_2 \}$. 
This is a relative $H$ stable submanifold of $\mathfrak{Z}$. 
On $H(\mathfrak{Z} \setminus \mathfrak{Z^{\prime}})$, write
\begin{align*}
    u(z,t)=\left(\begin{smallmatrix}
        1 & 0 & 0 & y_1 t & 0 & 0 \\
        0 & 1 & 0 & y_2 t & 0 & 0 \\
        0 & 0 & 1 & 0 & 0 & 0\\
        0 & 0 & 0 & 1 & 0 & 0\\
        0 & 0 & 0 & 0 & 1 & 0\\
        0 & 0 & 0 & 0 & 0 & 1
    \end{smallmatrix}\right)\quad {\rm and}\quad 
    v(z,t)=\left(\begin{smallmatrix}
        1 & 0 & 0 & 0 & 0 & 0 \\
        0 & 1 & 0 & 0 & 0 & 0 \\
        0 & 0 & 1 & 0 & 0 & 0\\
        w_1 t & w_2 t & 0 & 1 & 0 & 0\\
        0 & 0 & 0 & 0 & 1 & 0\\
        0 & 0 & 0 & 0 & 0 & 1
    \end{smallmatrix}\right).
\end{align*}
Then it is straightforward to see that 
\begin{align*}
    u(z,t)z=z v(z,t), \;  \mathrm{i.e.}, \; (u(z,t),v(z,t)^{-\tau})z=z.
\end{align*}
Since $w_2 \neq y_2$, we must have that 
\begin{align*}
    \chi(u(z,t),v(z,t)^{-\tau})=\psi(y_2 t- w_2 t) \neq 1,
\end{align*}
for suitable chosen $t\in F$.
By Theorem \ref{thm 2.1} and Remark \ref{rem .2}, we deduce that $\mathcal{D}^{'}_{\chi}(H\mathfrak{Z}\setminus H\mathfrak{Z}^{'})=0$.

On $H \mathfrak{Z^{\prime}}$, take $w_2 = y_2=b$. This submanifold has an $H$ slice $\mathfrak{Z}^{''}$ given by
\begin{align*}
\left(\begin{smallmatrix}
    x_{11} & 0 & 0 & y_1 & 0 & 0\\
    0 & x_{22} & 0 & b & 0 & 0\\
    0 & 0 & 0 & 0 & a & 0\\
    w_1 & b & 0 & 0 & 0 & 0\\
    0 & 0 & a & 0 & 0 & 0\\
    0 & 0 & 0 & 0 & 0 & 1
\end{smallmatrix}\right).
\end{align*}
It is easy to verify that $\mathfrak{Z}^{''}$ is stable under the subgroup $H_{\mathfrak{Z}^{''}}$ of $H$, which consists of elements of the form
\begin{align*}
    \left(\left(\begin{smallmatrix}
      t & 0 & 0 & 0 & 0 & 0 \\
        0 & 1 & 0 & 0 & 0 & 0 \\
        0 & 0 & t & 0 & 0 & 0\\
        0 & 0 & 0 & 1 & 0 & 0\\
        0 & 0 & 0 & 0 & t & 0\\
        0 & 0 & 0 & 0 & 0 & 1  
    \end{smallmatrix}\right),
    \left(\begin{smallmatrix}
      t^{-1} & 0 & 0 & 0 & 0 & 0 \\
        0 & 1 & 0 & 0 & 0 & 0 \\
        0 & 0 & t^{-1} & 0 & 0 & 0\\
        0 & 0 & 0 & 1 & 0 & 0\\
        0 & 0 & 0 & 0 & t^{-1} & 0\\
        0 & 0 & 0 & 0 & 0 & 1   
    \end{smallmatrix}\right)\right).
\end{align*}
Consider $\widetilde{H_{\mathfrak{Z}}}$ defined as before and let $\widetilde{\chi}_{\mathfrak{Z}^{''}}$ be the restriction of $\widetilde{\chi}$ to $\widetilde{H_{\mathfrak{Z}^{''}}}$. 
Write 
\begin{align*}
    \widetilde{F^{\times}}= \{1,\tau\}\ltimes F^{\times}
\end{align*}
with the semidirect product given by the action 
$\tau(t)=t^{-1}$ 
and let $\chi_{\widetilde{F^{\times}}}$ be the character on $\widetilde{F^{\times}}$ such that 
$\chi_{\widetilde{F^{\times}}}(F^{\times})=1$ and $\chi_{\widetilde{F^{\times}}}(\tau)=-1$.
Let $\widetilde{F^{\times}}$ act on $F^6=F^2 \times F^4$ by
    \begin{align*}
        t(y_1, w_1) =(ty_1, t^{-1}w_1), \; t(x_{11}, x_{22}, a, b)=(x_{11}, x_{22}, a, b)
    \end{align*}
and \begin{align*}
    \tau(y_1, w_1)=(w_1, y_1),  \; \tau(x_{11}, x_{22}, a, b)=(x_{11}, x_{22}, a, b).
\end{align*}
To prove that $\mathcal{D}^{'}_{\widetilde{\chi}_{\mathfrak{Z}^{''}}}(\mathfrak{Z}^{''})=0$, it is enough to prove that  
$\mathcal{D}^{'}_{\chi_{\widetilde{F^{\times}}}}(F^6)=0.$ 
Since the action of $\widetilde{F^{\times}}$ on $F^4$ is trivial, where $F^4$ is the second factor of $F^6=F^2 \times F^4$, 
it reduces to show that $\mathcal{D}^{'}_{\chi_{\widetilde{F^{\times}}}}(F^2)=0$ by Lemma \ref{lem-prod}, which is exactly \cite[Lemma 4.6]{MR2637582}.
Finally by Lemma \ref{lem 2.2}, we must have that 
\begin{align*}
       \mathcal{D}^{'}_{\widetilde{\chi}}(H \mathfrak{Z^{\prime}})=\mathcal{D}^{'}_{\widetilde{\chi}}(H \mathfrak{Z}^{''})=0.
\end{align*}
\end{proof}

The case of $i=2$ is given by 
\begin{lem}
    $\mathcal{D}^{'}_{\chi}(Z_{6,2}\setminus Z_{6,3})=0$.
\end{lem}

\begin{proof}
Following from the calculation in \cite[Lemma 6.4]{MR2763736}, for any 
    \begin{align*}
        x=\left(\begin{smallmatrix}
            x_{11} & x_{12} & x_{13} & 0 & 0 & 0 \\ x_{21} & x_{22} & x_{23} & 0 & 0 & 0 \\ x_{31}& x_{32}  & 0 & 0 & 0 & 0 \\ 0 & 0 & 0 & 1 & 0 & 0 \\ 0 &0& 0 & 0 & 0 & 1 \\ 0 & 0 & 0 & 0 & 1 & 0
        \end{smallmatrix}\right)\in \FZ_{6,2}^{\prime}\setminus \FZ_{6,3}^{\prime}
    \end{align*}
    and write 
    \begin{align*}
        u(x,t)=\left(\begin{smallmatrix}
            1 & 0 & x_{13}t & 0 & 0 & 0 \\ 0 & 1 & x_{23}t & 0 & 0 & 0 \\ 0 & 0  & 1 & 0 & 0 & 0 \\ 0 & 0 & 0 & 1 & 0 & 0 \\ 0 & 0 & 0 & 0 & 1 & 0 \\ 0 & 0 & 0 & 0 & 0 & 1
        \end{smallmatrix}\right)\quad\mathrm{and}\quad v(x,t)=\left(\begin{smallmatrix}
            1 & 0 & 0 & 0 & 0 & 0 \\ 0 & 1 & 0 & 0 & 0 & 0 \\ tx_{31} & tx_{32}  & 1 & 0 & 0 & 0 \\ 0 & 0 & 0 & 1 & 0 & 0 \\ 0 & 0 & 0 & 0 & 1 & 0 \\ 0 & 0 & 0 & 0 & 0 & 1
        \end{smallmatrix}\right).
    \end{align*}
    Then $u(x,t)x=xv(x,t)$. Since $x_{13}\neq x_{31}$, we have
    \begin{align*}
        \chi(u(x,t),v(x,t)^{-\tau})=\psi_F(x_{13}t-x_{31}t)\neq 1
    \end{align*}
    for a suitable chosen $t\in F$. Thus the lemma follows from Theorem \ref{thm 2.1} and Remark \ref{rem .2} as before.
\end{proof}

The case of $i=3$ is given by 
\begin{lem}
    $\mathcal{D}^{'}_{\widetilde{\chi}}(Z_{6,3}\setminus Z_{6,4})=0$.
\end{lem}

\begin{proof}
On $Z_{6,3} \setminus Z_{6,4}$, by Lemma 6.5 in \cite{MR2763736}, every matrix in $Z_{6,3} \setminus Z_{6,4}$  is in the same $H$ orbit as a matrix of the form
\begin{align*}
\left(\begin{smallmatrix}
    0 & 0 & a & 0 & 0 & 0\\
    0 & x_{22} & 0 & 0 & 0 & 0\\
    a & 0 & 0 & 0 & 1 & 0\\
    0 & 0 & 0 & 1 & 0 & 0\\
    0 & 0 & 0 & 0 & 0 & 1
\end{smallmatrix}\right).
\end{align*}
Since this matrix is symmetric, we can take $h$ to be the identity in Lemma \ref{lem 2.7}. Then each $H$ orbit in $Z_{6,3} \setminus Z_{6,4}$ is $\tau$-stable and any $(H,\chi)$ invariant distribution on this $H$ orbit will be $\tau$-invariant. In other words, every $(\widetilde{H},\widetilde{\chi})$-equivariant distribution on each $H$ orbit will vanish. By Remark \ref{rem 2.3}, 
we must have that $\mathcal{D}^{'}_{\widetilde{\chi}}(Z_{6,3}\setminus Z_{6,4})=0$.
\end{proof}

The case of $i=4$ is given by 
\begin{lem}
    $\mathcal{D}^{'}_{\chi}(Z_{6,4}\setminus Z_{6,5})=0$.
\end{lem}

\begin{proof}
  Following from the calculation in \cite[Lemma 6.6]{MR2763736}, for any 
    \begin{align*}
        x=\left(\begin{smallmatrix}
            x_{11} & 0 & 0 & 0 & 0 & 0 \\ 0 & 0 & x_{23} & 0 & 0 & 0 \\ 0 & x_{32}  & 0 & 0 & 0 & 0 \\ 0 & 0 & 0 & 0 & 1 & 0 \\ 0 & 0 & 0 & 1 & 0 & 0 \\ 0 & 0 & 0 & 0 & 0 & 1
        \end{smallmatrix}\right)\in \FZ_{6,4}^{''}\setminus \FZ_{6,5}^{''}
    \end{align*}
    and write 
    \begin{align*}
        u(x,t)=\left(\begin{smallmatrix}
            1 & 0 & x_{11}x_{32}^{-1}t & 0 & 0 & 0 \\ t & 1 & 0 & 0 & x_{23}t & 0 \\ 0 & 0  & 1 & 0 & 0  & 0 \\ 0 & 0 & t & 1 & 0& t \\ 0 & 0 & 0 & 0 & 1 & 0 \\ 0 & 0 & 0 & 0 & t & 1
        \end{smallmatrix}\right)\quad\mathrm{and}\quad v(x,t)=\left(\begin{smallmatrix}
            1 & t & 0 & 0 & 0 & 0 \\ 0 & 1 & 0 & 0 & 0 & 0 \\ x_{23}^{-1}x_{11}t & 0  & 1 & t & 0 & 0 \\ 0 & 0 & 0 & 1 & 0 & 0 \\ 0 & tx_{32} & 0 & 0 & 1 & t \\ 0 & 0 & 0 & t & 0 & 1
        \end{smallmatrix}\right).
    \end{align*}
    Then $u(x,t)x=xv(x,t)$, and $\chi(u(x,t),v(x,t)^{-\tau})=\chi_F((x_{32}^{-1}-x_{23}^{-1})x_{11}t)\neq 1$ for a suitable chosen $t\in F$. Then the lemma follows from Theorem \ref{thm 2.1}
    and Remark \ref{rem .2} as before.
\end{proof}

The case of $i=5$ is given by 
\begin{lem}
    $\mathcal{D}^{'}_{\widetilde{\chi}}(Z_{6,5})=0$.
\end{lem}

\begin{proof}
On $Z_{6,5}$, by the definition in \cite{MR2763736}, every matrix in $Z_{6,5}$ is in the same $H$ orbit as a matrix of the form
\begin{align*}
\left(\begin{smallmatrix}
    x_{11} & 0 & 0 & 0 & 0 & 0\\
    0 & 0 & a & 0 & 0 & 0\\
    0 & a & 0 & 0 & 1 & 0\\
    0 & 0 & 0 & 1 & 0 & 0\\
    0 & 0 & 0 & 0 & 0 & 1
\end{smallmatrix}\right).
\end{align*}
This is again a symmetric matrix. Thus, by the same argument, we must have that $\mathcal{D}^{'}_{\widetilde{\chi}}(Z_{6,5})=0$.
\end{proof}

Therefore, the assertion in \eqref{Z0-5} has been completely proved. This completes the proof of the assertion in \eqref{G43} and Proposition \ref{0=TG4}. 
Combining with Proposition \ref{0=TG4} with Proposition \ref{0=TG3}, we obtain that 
\begin{align}\label{G40}
    \mathcal{D}^{'}_{\widetilde{\chi}}(G^4_{\mathrm{open}})=0.
\end{align}

\section{Proof of Proposition \ref{0=TG6}}\label{sec-TG6}

We prove Proposition \ref{0=TG6}, namely, the following vanishing property
    \begin{align}\label{0=TG512}
        T|_{G^5_{\mathrm{open},1}}=0\quad {\rm and}\quad T |_{G^5_{\mathrm{open},2}}=0
    \end{align}
for any $T \in \D^{\prime}(G)_{\widetilde{\chi}}$, which is $\mathfrak{z}$-finite as in Theorem \ref{prop 3.2}. This will be done separately.

\subsection{The case of $G^{5}_{\mathrm{open},1}$}\label{section 3211}

Because of the vanishing in \eqref{G40}, to prove that $T|_{G^5_{\mathrm{open},1}}=0$, it suffices to prove that $\mathcal{D}^{'}_{\chi}(G^{5}_{\mathrm{open},1}\setminus G^4_{\mathrm{open}})=0$. 
We are going to use the same argument as that in Section \ref{section 4.4} by using the structure of wavefront set of $T$. Here is the analogy of Proposition \ref{prop 4.4} for the current case. 

\begin{prop}\label{lem 4.4}
    Let $T^{\prime} \in \mathcal{D}^{\prime}_{\chi}(G^5_{\mathrm{open},1})$ be a distribution such that
    \begin{align*}
        \wf_x(T^{\prime})\subset\mathcal{N}, \; \mathrm{for} \; \mathrm{all} \; x\in G^5_{\mathrm{open},1}.
    \end{align*}
    If $T^{\prime} |_{G^4_{\mathrm{open}}} =0 $, then $T^{\prime} =0$.
\end{prop}

Recall from \eqref{G4op} and \eqref{G5op1} that
$G^5_{\mathrm{open},1}\setminus G^4_{\mathrm{open}}=G_R$ with the rank-matrix $R=\begin{pmatrix}
    3 & 1 \\ 2 & 1
\end{pmatrix}$. 
The proof of Proposition \ref{lem 4.4} takes two steps as that of Proposition \ref{prop 4.4}. The first step is to understand the support of the distribution $T'$ as given in Proposition \ref{lem 4.4} and 
the second step is to apply the wavefront set condition to the possible support of $T'$. 

Note that every matrix in the particularly given $G_R$ is in the same $H$-orbit as an element in $\mathfrak{G}_R$ that is of the form: 
\begin{align*}
    g=\begin{pmatrix}
        x & 0 & 0 \\ 0 & 1_2 & 0 \\0 & 0 & y
    \end{pmatrix}\begin{pmatrix}
        E_{11} & E_{22} & 0 \\ 0 & E_{11} & E_{22} \\ E_{22} & 0 & E_{11}
    \end{pmatrix}\begin{pmatrix}
        z & 0 & 0 \\ 0 & 1_2 & 0 \\ 0 & 0 & w
    \end{pmatrix}.
\end{align*}
Write 
\begin{align*}
    \mathfrak{G}_R^{'}=\left\{ g= \begin{pmatrix}
        x & 0 & 0 \\ 0 & 1_2 & 0 \\0 & 0 & y
    \end{pmatrix}\begin{pmatrix}
        E_{11} & E_{22} & 0 \\ 0 & E_{11} & E_{22} \\ E_{22} & 0 & E_{11}
    \end{pmatrix}\begin{pmatrix}
        z & 0 & 0 \\ 0 & 1_2 & 0 \\ 0 & 0 & w
    \end{pmatrix} :\;y_{11}=z_{22}=0, \;y_{12}z_{21}=1         \right\}.
\end{align*}
On $H(\mathfrak{G}_R \setminus \mathfrak{G}_R^{'})$, by taking
\small{\begin{align*}
    u=\begin{pmatrix}
        1_2 & 0 & x\begin{pmatrix}
            0 & 0 \\ 0 & d_{22}
        \end{pmatrix}y^{-1}
     \\ 0 & 1_2 & \begin{pmatrix}
        0 & c_{12} \\ 0 &c_{22}
    \end{pmatrix}y^{-1} \\ 0 & 0 & 1_2\end{pmatrix}\quad {\rm and}\quad 
    v=\begin{pmatrix}
        1_2 & 0 & 0 \\ \begin{pmatrix}
            0 & c_{12} \\ 0 & d_{22}
        \end{pmatrix}z & 1_2 & 0 \\ w^{-1}\begin{pmatrix}
            0 & 0 \\ 0 & c_{22}
        \end{pmatrix}z & 0&  1_2
    \end{pmatrix}, 
\end{align*}}
we have that  $ug=gv$, i.e. $(u,v^{-\tau})g=g$. Since $g\notin \mathfrak{G}_R^{'}$,
\begin{align*}
    \chi(u,v^{-\tau})=\psi\left(-\frac{y_{12}}{\det y}c_{12}+\frac{y_{11}}{\det y} c_{11}-z_{21}c_{12}-z_{22}d_{22}\right)\neq 1,
\end{align*}
for suitable chosen $c_{12},c_{22}$ and $d_{22}$.
By Remark \ref{rem .2}, we educe that 
\begin{align}\label{GR-R}
    \mathcal{D}^{\prime}_{\chi}(H(\mathfrak{G}_R \setminus \mathfrak{G}_R^{'}))=0.
\end{align}

Let $T^{\prime} \in \mathcal{D}^{\prime}_{\chi}(G^5_{\mathrm{open},1})$ be distribution which vanishes on $G^4_{\mathrm{open}}$ as in Proposition \ref{lem 4.4}. Consider 
\[
T^{\prime} |_{G^4_{\mathrm{open}} \sqcup H(\mathfrak{G}_R \setminus \mathfrak{G}_R^{'})},
\]
which is the restriction of $T^{\prime}$ to the open submanifold $G^4_{\mathrm{open}} \sqcup H(\mathfrak{G}_R \setminus \mathfrak{G}_R^{'})$ of $G^5_{\mathrm{open},1}$. Since $T^{\prime}$ vanishes on $G^4_{\mathrm{open}}$, by the exact sequence in \eqref{2.4} and the vanishing in \eqref{GR-R}, we obtain that 
\begin{align*}
    T^{\prime} |_{G^4_{\mathrm{open}} \sqcup H(\mathfrak{G}_R \setminus \mathfrak{G}_R^{'})} \in \mathcal{D}^{\prime}_{\chi}(H(\mathfrak{G}_R \setminus \mathfrak{G}_R^{'}))=0.
\end{align*}
It follows that  $T^{\prime}$ must be supported in the closed submanifold $H\mathfrak{G}_R^{\prime}$.

In order to prove that $T^{\prime} =0$ on $H\mathfrak{G}_R^{\prime}$, we use the wavefront set condition. We write $\mathfrak{G}_R^{''}$ to be the set of elements in the particularly given $G_R$ of the following form:
\begin{align}\label{4.1}
    x=\left(\begin{smallmatrix}
        x_{11} & x_{12} & 0 & x_{14} & 0 & 0 \\x_{21} & x_{22} & 0 & x_{24} & 0 & 0 \\ 0 & 0& 1 & 0 & 0 & 0 \\0 & 0  & 0 & 0 &x_{45} & x_{46} \\1 & 0 & 0 & 0 & 0 & 0 \\0 & 0 & 0 & 0& x_{65} & x_{66}.
    \end{smallmatrix}\right).
\end{align}
One checks by direct computation that $H\mathfrak{G}_R^{'}=H\mathfrak{G}_R^{''}$.

Assume that $T^{\prime} \neq0 $ on $H\mathfrak{G}_R^{\prime}$. Since $T^{\prime}$ is $H$-equivariant, there exists $x \in \mathfrak{G}_R^{''}$ as in \eqref{4.1} such that  $x \in \supp(T^{\prime})$. Then from the definition of the wave front set, we have
$(x,0) \in \wf_x(T^{\prime}).$ 
By Theorem \ref{thm 2.12}, this implies the conormal bundles have the property:
\begin{align*}
    \cn^{G^5_{\mathrm{open},1}}_{H\mathfrak{G}_R^{'},x}=\cn^G_{H\mathfrak{G}_R^{'},x} \subset \wf_x(T^{\prime}).
\end{align*}
Take 
\begin{align*}
    x'=\left(\begin{smallmatrix}
        0 & 0 & 0 & 0 & 1 & 0 \\0 & 0 & 0 & 0 & 1 & 0 \\ 0 & 0& -1 & 0 & 0 & 0 \\0 & 0  & -1 & 0 &0 & 0\\0 & 0 & 0 & 0 & 0 & 0 \\0 & 0 & 0 & 0& 0 & 0
    \end{smallmatrix}\right). 
\end{align*}
We have that 
\begin{align*}
    x^{'}x=\left(\begin{smallmatrix}
         1 & 0 & 0 & 0 & 0 & 0 \\1 & 0 & 0 & 0 & 0 & 0 \\ 0 & 0& -1 & 0 & 0 & 0 \\0 & 0  & -1 & 0 &0 & 0\\0 & 0 & 0 & 0 & 0 & 0 \\0 & 0 & 0 & 0& 0 & 0
    \end{smallmatrix}\right)\quad {\rm and}\quad 
    xx^{'}=\left(\begin{smallmatrix}
        0 & 0 & -x_{14} & 0 & x_{11}+x_{12} & 0\\
        0 & 0 & -x_{24} & 0 & x_{21}+x_{22} & 0\\
        0 & 0 & -1 & 0 & 0 & 0\\
        0 & 0 & 0 & 0 & 0 & 0\\
        0 & 0 & 0 & 0 & 1 & 0\\
         0 & 0 & 0 & 0 & 0 & 0
    \end{smallmatrix}\right),
\end{align*}
which shows 
\begin{align*}
    T_x(H\mathfrak{G}_R^{''})=T_x(\mathfrak{G}_R^{''})+\Lie(S)x+x(\Lie(S))^{\tau} \subset (Fx^{\prime})^{\bot}.
\end{align*}
Note $x^{\prime}x$ is not nilpotent. As before, this contradicts the assumption of Proposition \ref{lem 4.4} on the structure of the wavefront set of $T'$. 
Therefore we must have that $T^{\prime}=0$. We are done.

\subsection{The case of  $G^{5}_{\mathrm{open},2}$}\label{ssec-G52}
Since $G^{5}_{\mathrm{open},2}\setminus G^4_{\mathrm{open}}=G_R$ with the rank-matrix $R=\begin{pmatrix}
    3 & 2 \\ 1 & 1
\end{pmatrix}$, which is the transpose of the rank-matrix $\begin{pmatrix}
    3 & 1 \\ 2 & 1
\end{pmatrix}$. Then a parallel argument as in Section \ref{section 3211} shows that the following proposition holds.
\begin{prop}
    Let $T^{\prime} \in \mathcal{D}^{\prime}_{\chi}(G^5_{\mathrm{open},2})$ be a distribution such that
    \begin{align*}
        \wf_x(T^{\prime})\subset\mathcal{N}, \; \mathrm{for} \; \mathrm{all} \; x\in G^5_{\mathrm{open},2}.
    \end{align*}
    If $T^{\prime} |_{G^4_{\mathrm{open}}} =0 $, then $T^{\prime} =0$.
\end{prop}

Combining with Proposition \ref{lem 4.4}, we complete the proof of Proposition \ref{0=TG6}.

\section{Proof of Proposition \ref{0=TG-6}}\label{sec-G6}

We prove Proposition \ref{0=TG-6}, namely, the following vanishing property
\begin{align}\label{0=G-6}
    \mathcal{D}^{\prime}_{\widetilde{\chi}}(G \setminus G^6_{\mathrm{open}})=0.
\end{align}
Recall from Section \ref{ssec-SG}, we have 
\begin{align}\label{G-G6}
    G \setminus G^6_{\mathrm{open}}
=
\bigsqcup_R G_R
\end{align}
with $R$ running through all rank-matrices in Section \ref{ssec-SG} except the following six
\begin{align*}
     \begin{pmatrix}
        4 & 2 \\ 2 & 2
    \end{pmatrix},\begin{pmatrix}
        4 & 2 \\ 2 & 1
    \end{pmatrix},\begin{pmatrix}
        3 & 2 \\ 2 & 2
    \end{pmatrix},\begin{pmatrix}
        3 & 2 \\ 2 & 1
    \end{pmatrix}, \begin{pmatrix}
        3 & 2 \\ 1 & 1
    \end{pmatrix}, \begin{pmatrix}
        3 & 1 \\ 2 & 1
    \end{pmatrix}. 
\end{align*}
In order to verify the assertion in \eqref{0=G-6}, it suffices to show that 
there is no $(\widetilde{H},\widetilde{\chi})$-equivariant distribution on $G_{R}$ for each $G_R$ in \eqref{G-G6}. We will do this case by case.

\subsubsection{Case:\ $R=\begin{pmatrix}
    4 & 2 \\ 2 & 0
\end{pmatrix}$}\
In this case, every matrix in the given $G_R$ is in the same $H$-orbit as elements in $\mathfrak{G}_R$ given by
\begin{align*}
    g=\begin{pmatrix}
    1_2 & 0 & 0\\ 0 & 0 & x \\ 0 & y & 0
\end{pmatrix},\;x, y\in\GL_2(F). 
\end{align*}
Write
\begin{align*}
    \mathfrak{G}_R^{'}=\left\{g=\begin{pmatrix}
    1_2 & 0 & 0\\ 0 & 0 & x \\ 0 & y & 0
\end{pmatrix}\in\mathfrak{G}_R \colon x= y   \right\}.
\end{align*}
On $H(\mathfrak{G}_R\setminus\mathfrak{G}_{R}^{'})$, write
\begin{align*}
    u=\begin{pmatrix}
        1_2 & 0 & 0 \\ 0 & 1_2 & c \\ 0 & 0 &1_2
    \end{pmatrix}\quad {\rm and}\quad 
    v=\begin{pmatrix}
        1_2 & 0 & 0 \\ 0 & 1_2 & 0 \\ 0 & x^{-1}cy &1_2
    \end{pmatrix}.
\end{align*}
It is easy to check that 
$ug=gv$, i.e. $(u,v^{-\tau})g=g$. Since $g\notin \mathfrak{G}_R^{'}$,
\begin{align*}
    \chi(u,v^{-\tau})=\psi_{F}(\Tr(c))\psi_{F}(\Tr(-x^{-1}cy))=\psi_{F}(\Tr(1_2-yx^{-1})c),
\end{align*}
which is not trivial for suitable $c$.
By Remark \ref{rem .2}, we must have that $\mathcal{D}^{'}_{\chi}(H\mathfrak{G}_R\setminus H\mathfrak{G}_R^{'})=0$.

On $H\mathfrak{G}_R^{'}$, we take $(h_1,h_2)=(1_6,1_6)$. It is clear that 
$h_1g=g^{\tau}h_2$, i.e. $(h_1,h_2^{-\tau})g=g^{\tau}$ and $\chi(h_1,h_2^{-\tau})=1$. 
By Lemma \ref{lem 2.7}, the $H$-orbit $Hg$ is $\tau$-invariant and any $(H,\chi)$-equivariant distribution on this $H$ orbit is $\tau$-invariant. In other words, there is no non-zero $(\widetilde{H},\widetilde{\chi})$-equivariant distribution on this orbit. By Remark \ref{rem 2.3}, we must have that $\mathcal{D}^{'}_{\widetilde{\chi}}(H\mathfrak{G}_R^{'})=0$.

\subsubsection{Case:\ $R=\begin{pmatrix}
    2 & 0 \\ 0 & 0
\end{pmatrix}$}\
In this case, every matrix in the given $G_R$ is in the same $H$-orbit as elements in $\mathfrak{G}_R$ given by
\begin{align*}
    g=\begin{pmatrix}
    0 & 0 & x\\ 0 & 1_2 & 0 \\ y & 0 & 0
\end{pmatrix},\;x, y\in\GL_2(F).  
\end{align*}
Write
\begin{align*}
    \mathfrak{G}_R^{'}=\left\{\begin{pmatrix}
    0 & 0 & 1_2\\ 0 & 1_2 & 0 \\ 1_2 & 0 & 0
\end{pmatrix}   \right\} \subset \mathfrak{G}_R.
\end{align*}
On $H(\mathfrak{G}_R\setminus\mathfrak{G}_{R}^{'})$, write
\begin{align*}
    u=\begin{pmatrix}
        1_2 & b & 0 \\ 0 & 1_2 & c \\ 0 & 0 &1_2
    \end{pmatrix}\quad {\rm and}\quad 
    v=\begin{pmatrix}
        1_2 & 0 & 0 \\ cy & 1_2 &0 \\ 0 & x^{-1}b &1_2
    \end{pmatrix}.
\end{align*}
It is clear that $ug=gv$, i.e. $(u,v^{-\tau})g=g$. 
Then we have that 
\begin{align*}
    \chi(u,v^{-\tau})=\psi_{F}(\Tr((1_2-y)c+ (1_2-x^{-1})b)),
\end{align*}
which is not trivial for suitable $b$ and $c$. 
By Remark \ref{rem 2.3}, we deduce that $\mathcal{D}^{'}_{\chi}\left(H \left( \mathfrak{G}_R\setminus\mathfrak{G}_R^{'} 
 \right)   
 \right)=0$.

On $H\mathfrak{G}_R^{'}$, we take that $(h_1,h_2)=(1_6,1_6)$. Then  
$h_1g=g^{\tau}h_2$, i.e. $(h_1,h_2^{-\tau})g=g^{\tau}$ 
and
$\chi(h_1,h_2^{-\tau})=1$. 
By Lemma \ref{lem 2.7}, the $H$-orbit $Hg$ is $\tau$-invariant and any $(H,\chi)$-equivariant distribution on this $H$ orbit is $\tau$-invariant. In other words, there is no non-zero $(\widetilde{H},\widetilde{\chi})$-equivariant distribution on this orbit. By Remark \ref{rem 2.3}, we deduce that $\mathcal{D}^{'}_{\widetilde{\chi}}(H\mathfrak{G}_R^{'})=0$.

\subsubsection{Case:\ $R=\begin{pmatrix}
    2 & 2 \\ 0 & 0
\end{pmatrix}$ or $\begin{pmatrix}
    2 & 0 \\ 2 & 0
\end{pmatrix}$
}\ 
The two cases can be treated in the same way. We only provide the details for $G_R$ with $R=\begin{pmatrix}
    2 & 2 \\ 0 & 0
\end{pmatrix}$. 
In this case, every matrix in the given $G_R$ is in the same $H$-orbit as elements in $\mathfrak{G}_R$ given by
\begin{align*}
    g=\begin{pmatrix}
    0 & x & 0\\ 0 & 0 & 1_2 \\ y & 0 & 0
\end{pmatrix},\;x, y\in\GL_2(F). 
\end{align*}
Write
\begin{align*}
    u=\begin{pmatrix}
        1_2 & 0 & 0 \\ 0 & 1_2 & c \\ 0 & 0 &1_2
    \end{pmatrix}\quad {\rm and}\quad 
    v=\begin{pmatrix}
        1_2 & 0 & 0 \\ 0 & 1_2 & 0 \\ cy & 0 &1_2
    \end{pmatrix}.
\end{align*}
It is easy to see that 
$ug=gv$, i.e. $(u,v^{-\tau})g=g$ and $\chi(u,v^{-\tau})=\psi_{F}(\Tr(c))$, 
which is not trivial for suitable $c$. Then by Remark \ref{rem 2.3}, we must have that $\mathcal{D}^{'}_{\chi}(G_R)=0$.

\subsubsection{Case:\ $R=\begin{pmatrix}
    2 & 2 \\ 2 & 2
\end{pmatrix}$}\ 
In this case, every matrix in the given $G_R$ is in the same $H$-orbit as elements in $\mathfrak{G}_R$ given by
\begin{align*}
    g=\begin{pmatrix}
    0 & x & 0\\ y & 0 & 0 \\ 0 & 0 & 1_2
\end{pmatrix},\;x, y\in\GL_2(F). 
\end{align*}
Write
\begin{align*}
    \mathfrak{G}_R^{'}=\left\{g=\begin{pmatrix}
    0 & x & 0\\ y & 0 & 0 \\ 0 & 0 & 1_2
\end{pmatrix}\in\mathfrak{G}_R \colon x= y   \right\}.
\end{align*}
On $H(\mathfrak{G}_R\setminus\mathfrak{G}_{R}^{'})$, write
\begin{align*}
    u=\begin{pmatrix}
        1_2 & b & 0 \\ 0 & 1_2 & 0 \\ 0 & 0 &1_2
    \end{pmatrix}\quad {\rm and}\quad     v=\begin{pmatrix}
        1_2 & 0 & 0 \\ x^{-1}by & 1_2 & 0 \\ 0 & 0 &1_2
    \end{pmatrix}.
\end{align*}
It can be checked by direct computation that 
$ug=gv$, i.e. $(u,v^{-\tau})g=g$. 
Since $g\notin \mathfrak{G}_R^{'}$,
\begin{align*}
    \chi(u,v^{-\tau})=\psi_{F}(\Tr(b))\psi_{F}(\Tr(-x^{-1}by))=\psi_{F}(\Tr(1_2-yx^{-1})b),
\end{align*}
which is not trivial for suitable $b$. 
By Remark \ref{rem .2}, we must have that $\mathcal{D}^{'}_{\chi}(H\mathfrak{G}_R\setminus H\mathfrak{G}_R^{'})=0$.

On $H\mathfrak{G}_R^{'}$, we take that $(h_1,h_2)=(1_6,1_6)$. Then  
$h_1g=g^{\tau}h_2$ i.e.$(h_1,h_2^{-\tau})g=g^{\tau}$
and
$\chi(h_1,h_2^{-\tau})=1$. 
By Lemma \ref{lem 2.7}, the $H$-orbit $Hg$ is $\tau$-invariant and any $(H,\chi)$-equivariant distribution on this $H$ orbit is $\tau$-invariant. In other words, there is no non-zero $(\widetilde{H},\widetilde{\chi})$-equivariant distribution on this orbit. By Remark \ref{rem 2.3}, we deduce that $\mathcal{D}^{'}_{\widetilde{\chi}}(H\mathfrak{G}_R^{'})=0$.

\subsubsection{Case:\ $R=\begin{pmatrix}
    2 & 2 \\ 1 & 1
\end{pmatrix}$ or $\begin{pmatrix}
    2 & 1 \\ 2 & 1
\end{pmatrix}$
}\ 
The two cases can be treated in the same way. We only provide the details for $G_R$ with $R=\begin{pmatrix}
    2 & 2 \\ 1 & 1
\end{pmatrix}$.
In this case, every matrix in the given $G_R$ is in the same $H$-orbit as elements in $\mathfrak{G}_R$ given by
\begin{align*}
    g=\begin{pmatrix}
    x & 0 & 0\\ 0 & 1_2 & 0 \\ 0 & 0 & y
\end{pmatrix}x_0   \begin{pmatrix}
    z & 0 & 0\\ 0 & 1_2 & 0 \\ 0 & 0 & w
\end{pmatrix}, \quad {\rm with}\quad x_0=\begin{pmatrix}
        0  & 1_2 & 0 \\ E_{22} & 0 & E_{11} \\ E_{11} & 0 & E_{22}
    \end{pmatrix},
\end{align*}
where $x, y,z,w\in\GL_2(F)$
and  $E_{ij}(1\leq i,j\leq2)$ to be the $2\times 2$ matrix with entry $1$ at $(i,j)$ place and $0$ otherwise.
Write
\small{\begin{align*}
    u=\begin{pmatrix}
        1_2 & 0 & \begin{pmatrix}
            d_{11} & 0 \\
            d_{21} & 0
        \end{pmatrix}y^{-1} \\ 0 & 1_2 & 0
        \\ 0 & 0 &1_2
    \end{pmatrix}\quad {\rm and }\quad 
    v=\begin{pmatrix}
        1_2 & 0 & 0 \\ x^{-1}\begin{pmatrix}
            d_{11} & 0 \\
            d_{21} & 0
        \end{pmatrix}z & 1_2 & 0 \\ 0 & 0 &1_2
    \end{pmatrix},
\end{align*}}
We deduce easily that $ug=gv$, i.e. $(u,v^{-\tau})g=g$     
and that 
\begin{align*}
    \chi(u,v^{-\tau})=\psi_{F}(-\mathrm{Tr}( x^{-1}\begin{pmatrix}
            d_{11} & 0 \\
            d_{21} & 0
        \end{pmatrix}z))=\psi_{F}(-(zx^{-1})_{11}d_{11}-(zx^{-1})_{12}d_{21}),
\end{align*}
which is not trivial for suitably chosen $d_{11}$ and $d_{21}$, because $(zx^{-1})_{11}$ and $(zx^{-1})_{12}$ can not both be $0$, where $(zx)^{-1}_{11}$ and $(zx^{-1})_{12}$ are the $(1,1)$-entry and $(1,2)$-entry of $zx^{-1}$ respectively. By Remark \ref{rem 2.3}, we must have that $\mathcal{D}^{'}_{\chi}(G_R)=0$.

\subsubsection{Case:\ $R=\begin{pmatrix}
    2 & 1 \\ 0 & 0
\end{pmatrix}$ or $\begin{pmatrix}
    2 & 0 \\ 1 & 0
\end{pmatrix}$
}\ 
The two cases can be treated in the same way. We only provide the details for $G_R$ with $R=\begin{pmatrix}
    2 & 1 \\ 0 & 0
\end{pmatrix}$
In this case, every matrix in the given $G_R$ is in the same $H$-orbit as elements in $\mathfrak{G}_R$ given by
\begin{align*}
    g=\begin{pmatrix}
    x & 0 & 0\\ 0 & y & 0 \\ 0 & 0 & 1_2
\end{pmatrix}x_0   \begin{pmatrix}
    z & 0 & 0\\ 0 & w & 0 \\ 0 & 0 & 1_2
\end{pmatrix},\quad  {\rm with}\quad x_0=\begin{pmatrix}
        0  & E_{11} & E_{22} \\ 0 & E_{22} & E_{11} \\ 1_2 & 0 & 0
    \end{pmatrix}
\end{align*}
where $x, y,z,w\in\GL_2(F)$. 
Write
\small{\begin{align*}
    u=\begin{pmatrix}
        1_2 & 0 & 0 \\ 0 & 1_2 & y\begin{pmatrix}
            c_{11} & c_{12} \\ 0 & 0
        \end{pmatrix} \\ 0 & 0 &1_2
    \end{pmatrix}\quad {\rm and}\quad 
    v=\begin{pmatrix}
        1_2 & 0 & 0 \\ 0 & 1_2 & 0 \\ \begin{pmatrix}
            c_{11} & c_{12} \\ 0 & 0
        \end{pmatrix}z & 0 &1_2
    \end{pmatrix}.
\end{align*}}
One can easily check that $ug=gv$, i.e. $(u,v^{-\tau})g=g,$
and that 
\begin{align*}
    \chi(u,v^{-\tau})=\psi_{F}(y_{11}c_{11}+y_{21}c_{12}) \neq 1
\end{align*}
 for suitably chosen $c_{11}, c_{12} \in F$.  By Remark \ref{rem 2.3}, we must have that $\mathcal{D}^{'}_{\chi}(G_R)=0$, by the same argument.

\subsubsection{Case:\ $R=\begin{pmatrix}
    2 & 1 \\ 1 & 1
\end{pmatrix}$}\ 
In this case, every matrix in the given $G_R$ is in the same $H$ orbit as elements in $\mathfrak{G}_R$, given by 
\begin{align*}
    g=\begin{pmatrix}
        x  & 0 & 0 \\ 0 & y & 0 \\ 0 & 0 & 1_2
    \end{pmatrix}x_0\begin{pmatrix}
        z  & 0 & 0 \\ 0 & w & 0 \\ 0 & 0 & 1_2
    \end{pmatrix},\quad {\rm with }\quad x_0=\begin{pmatrix}
        0  & E_{11} & E_{22} \\ E_{11} & E_{22} & 0 \\ E_{22} & 0 & E_{11}
    \end{pmatrix}, 
\end{align*}
where $ x,y,z,w \in \GL_2(F).$ 
Take $\mathfrak{G}_R^{\prime}$ to be the subset of the given $\mathfrak{G}_R$ that consists of elements $g\in \mathfrak{G}_R$ with the property that 
$(y^{-1}x)_{11}=(zw^{-1})_{11}$, $(y^{-1}x)_{12}=0$, $(y^{-1}x)_{22}=w_{22}$, $(zw^{-1})_{21}=0$, and $(zw^{-1})_{22}=y_{22}$. 
On the submanifold, $H(\mathfrak{G}_R \setminus \mathfrak{G}_R^{\prime})$, write
\small{\begin{align*}
    u=\begin{pmatrix}
        1_2 & x\begin{pmatrix}
            b_{11} & 0\\
            b_{21} & b_{22}
        \end{pmatrix}y^{-1} & x\begin{pmatrix}
            0 & d_{12} \\ 0 & d_{22}
        \end{pmatrix}\\
        0 & 1_2 & y\begin{pmatrix}
            0 & 0\\
            0 & c_{22}
        \end{pmatrix}\\
        0 & 0 & 1_2
    \end{pmatrix}\quad {\rm and}\quad 
    v= \begin{pmatrix}
        1_2 & 0 & 0\\
        w^{-1}\begin{pmatrix}
            b_{11} & d_{12}\\
            0 & c_{22}
        \end{pmatrix}z & 1_2 & 0\\
          \begin{pmatrix}
            0 & 0\\
            b_{21} & d_{22}
        \end{pmatrix}z & \begin{pmatrix}
            0 & 0\\
            0 & b_{22}
        \end{pmatrix}w & 1_2
    \end{pmatrix}.
\end{align*}}
Then we have that $ug=gv, \; \mathrm{i.e.} \; (u,v^{-\tau})g=g.$
Since $g \notin \mathfrak{G}_R^{\prime}$, we obtain that 
\begin{align*}
    \chi(u,v^{-\tau})&=\psi((y^{-1}x)_{11}b_{11}+(y^{-1}x)_{12}b_{21}+(y^{-1}x)_{22}b_{22}+y_{22}c_{22}\\
    &\qquad -(zw^{-1})_{11}b_{11}
    -(zw^{-1})_{21}d_{12}-(zw^{-1})_{22}c_{22}-w_{22}b_{22}) \neq 1.
\end{align*}
for suitably chosen $b_{11},b_{21},b_{22},c_{22} \in F$.
By Remark \ref{rem .2}, we obtain that $\mathcal{D}^{'}_{\chi}\left(H\left(\mathfrak{G}_R\setminus\mathfrak{G}_R^{'}\right)\right)=0$.

On the submanifold $H \mathfrak{G}_R^{\prime}$, write
\begin{align*}
    h_1=\left(\begin{smallmatrix}
        a & z^{\tau}\left(\begin{smallmatrix}
        0 & (y^{\tau}a^{\tau} w^{-1})_{12} \\ 0 & 0
    \end{smallmatrix}\right)y^{-1} & z^{\tau}\left(\begin{smallmatrix}
        0 & 0 \\ a_{12} & 0
    \end{smallmatrix}\right) \\ 0 & a & w^{\tau}\left(\begin{smallmatrix}
        0 & (x^{\tau}a^{\tau}z^{-1})_{12} \\
        0 & 0
    \end{smallmatrix}\right) \\ 0 & 0 & a
    \end{smallmatrix} \right)\ {\rm and}\ 
    h_2=h_1^{\tau}=\left(\begin{smallmatrix}
        a^{\tau} & 0 & 0 \\ y^{-\tau}\left(\begin{smallmatrix}
            0 & 0 \\
            (w^{-\tau}ay)_{21} & 0
        \end{smallmatrix}\right)z & a^{\tau} & 0\\ \left(\begin{smallmatrix}
        0 & a_{12} \\ 0 & 0
    \end{smallmatrix}\right)z & \left(\begin{smallmatrix}
        0 & 0 \\
        (z^{-\tau}ax)_{21} & 0
    \end{smallmatrix}\right)w & a^{\tau}
    \end{smallmatrix}\right),
\end{align*}
where
$    a=\begin{pmatrix}
        1 & a_{12} \\ 0 & 1
    \end{pmatrix}, \; \mathrm{with} \; a_{12}=\frac{w_{21}}{w_{22}}-\frac{y_{12}}{y_{22}}$. 
Then have that $h_1g=g^{\tau}h_2$, i.e. $(h_1,h_2^{-\tau})g=g^{\tau}.$
It is clear that $\chi(h_1,h_2^{-\tau})=1$. 
By Lemma \ref{lem 2.7}, the $H$ orbit $Hg$ is $\tau$-invariant and any $(H,\chi)$-equivariant distribution on this $H$ orbit is $\tau$-invariant. In other words, there is no non-zero $(\widetilde{H},\widetilde{\chi})$-equivariant distribution on this orbit. By Remark \ref{rem 2.3}, we deduce that $\mathcal{D}^{'}_{\widetilde{\chi}}(H\mathfrak{G}_R^{'})=0$.

\subsubsection{Case:\ $R=\begin{pmatrix}
    3 & 1 \\ 1 & 0
\end{pmatrix}$}\ 
In this case, every matrix in the given $G_R$ is in the same $H$ orbit as elements in $\mathfrak{G}_R$, given by 
\begin{align*}
    g=\begin{pmatrix}
        x  & 0 & 0 \\ 0 & y & 0 \\ 0 & 0 & 1_2
    \end{pmatrix}x_0\begin{pmatrix}
        z  & 0 & 0 \\ 0 & w & 0 \\ 0 & 0 & 1_2
    \end{pmatrix},\quad {\rm with}\quad x_0=\begin{pmatrix}
        E_{11}  & 0 & E_{22} \\ 0 & E_{22} & E_{11} \\ E_{22} & E_{11} & 0
    \end{pmatrix},
\end{align*}
where $ x,y,z,w \in \GL_2(F).$ 
Write
\begin{align*}
    \mathfrak{G}_R^{\prime}=\{ 
    g \in \mathfrak{G}_R \colon y_{21}=w_{12}=0, \; y_{11}=w_{11} ,\;(zw^{-1})_{22}=y_{22} \; \mathrm{and} \; (y^{-1}x)_{22}=w_{22}.
    \}.
\end{align*}
On the submanifold $H(\mathfrak{G}_R \setminus \mathfrak{G}_R^{\prime})$, write
\small{\begin{align*}
    u=\begin{pmatrix}
        1_2 & x\begin{pmatrix}
            0 & 0\\
            0 & b_{22}
        \end{pmatrix}y^{-1} & x\begin{pmatrix}
            0 & 0\\ d_{21} & d_{22}
        \end{pmatrix}\\
        0 & 1_2 & y\begin{pmatrix}
            c_{11} & c_{12} \\
            0 & c_{22}
        \end{pmatrix}\\
        0 & 0 & 1_2
    \end{pmatrix}\quad {\rm and }\quad 
    v= \begin{pmatrix}
        1_2 & 0 & 0\\
        w^{-1}\begin{pmatrix}
            0 & 0\\
            0 & c_{22}
        \end{pmatrix}z & 1_2 & 0\\
          \begin{pmatrix}
            0 & c_{12}\\
            0 & d_{22}
        \end{pmatrix}z & \begin{pmatrix}
            c_{11} & 0\\
            d_{21} & b_{22}
        \end{pmatrix}w & 1_2
    \end{pmatrix}.
\end{align*}}
Then we have that $ug=gv, \; \mathrm{i.e.} \; (u,v^{-\tau})g=g.$
Since $g \notin \mathfrak{G}_R^{\prime}$,
\begin{align*}
    \chi(u,v^{-\tau})&=\psi((y^{-1}x)_{11}b_{22}+y_{11}c_{11}+y_{21}c_{12}+y_{22}c_{22}-(zw^{-1})_{22}c_{22}-c_{11}w_{11}-d_{21}w_{12}-b_{22}w_{22}) \neq 1.
\end{align*}
for suitably chosen $b_{22},c_{11},c_{12},c_{22}$ and $d_{21}\in F$. By Remark \ref{rem .2}, we get that $\mathcal{D}^{'}_{\chi}\left(H\left(\mathfrak{G}_R\setminus\mathfrak{G}_R^{'}  
 \right)\right)=0$.

On the submanifold $H \mathfrak{G}_R^{\prime}$, write
\small{\begin{align*}
    h_1=\begin{pmatrix}
        a & \begin{pmatrix}
            a_{12}z_{21} & 0 \\ a_{12}z_{22} & 0
        \end{pmatrix}y^{-1} & z^{\tau}\begin{pmatrix}
        0 & (x^{\tau}\alpha z^{-1})_{12} \\ 0 & 0
    \end{pmatrix} \\ 0 & a & w^{\tau}\begin{pmatrix}
        0 & 0 \\ (y^{\tau}\alpha w^{-1})_{21} & 0
    \end{pmatrix}  \\ 0 & 0 & a
    \end{pmatrix} ,
\end{align*}}
and
\small{\begin{align*}
    h_2=h_1^{\tau}=\begin{pmatrix}
        a^{\tau} &0 & 0 \\  y^{-\tau}\begin{pmatrix}
            a_{12}z_{21} & a_{12}z_{22} \\ 0 & 0
        \end{pmatrix} & a^{\tau} & 0\\ \begin{pmatrix}
            0 & 0 \\ (z^{-\tau}ax)_{21} & 0
        \end{pmatrix}z & \begin{pmatrix}
            0 & (w^{-\tau}ay)_{12} \\ 0 & 0
        \end{pmatrix}w  & a^{\tau}
    \end{pmatrix},
\end{align*}}
where $a=\begin{pmatrix}
        1 & a_{12}\\ 0 & 1
    \end{pmatrix}$, with $a_{12}=\frac{z_{21}-x_{12}}{x_{22}}$. 
Then we have that $h_1g=g^{\tau}h_2$, i.e. $(h_1,h_2^{-\tau})g=g^{\tau}.$
It is clear that $\chi(h_1,h_2^{-\tau})=1$
By Lemma \ref{lem 2.7}, the $H$ orbit $Hg$ is $\tau$-invariant and any $(H,\chi)$-equivariant distribution on this $H$ orbit is $\tau$-invariant. In other words, there is no non-zero $(\widetilde{H},\widetilde{\chi})$-equivariant distribution on this orbit. By \ref{rem 2.3}, we obtain that there is no $(\widetilde{H},\widetilde{\chi})$-equivariant distribution on $H \mathfrak{G}_R^{\prime}$.

\subsubsection{ Case:\ $R=\begin{pmatrix}
    3 & 1 \\ 2 & 0
\end{pmatrix}$ or $R=\begin{pmatrix}
    3 & 2 \\ 1 & 0
\end{pmatrix}$
}\ 
The two cases can be treated in the same way. We only provide the details for $G_R$ with $R=\begin{pmatrix}
    3 & 1 \\ 2 & 0
\end{pmatrix}$.
In this case, every matrix in the given $G_R$ is in the same $H$-orbit as elements in $\mathfrak{G}_R$ given by
\begin{align*}
   g=\begin{pmatrix}
    x & 0 & 0\\ 0 & y & 0 \\ 0 & 0 & 1_2
\end{pmatrix}x_0   \begin{pmatrix}
    z & 0 & 0\\ 0 & w & 0 \\ 0 & 0 & 1_2
\end{pmatrix}, \quad {\rm with}\quad x_0=\begin{pmatrix}
        E_{11}  & E_{22} & 0 \\ 0 & 0 & 1_2 \\ E_{22} & E_{11} & 0
    \end{pmatrix},
\end{align*}
where $x, y,z,w\in\GL_2(F)$. 
Write
\begin{align*}
    u=\begin{pmatrix}
        1_2 & 0 & 0 \\ 0 & 1_2 & y\begin{pmatrix}
            0 & c_{12} \\ 0 & c_{22}
        \end{pmatrix} \\ 0 & 0 &1_2
    \end{pmatrix}
\quad {\rm and}\quad 
    v=\begin{pmatrix}
        1_2 & 0 & 0 \\ 0 & 1_2 & 0 \\ \begin{pmatrix}
            0 & c_{12} \\ 0 & c_{22}
        \end{pmatrix}z & 0 &1_2
    \end{pmatrix}.
\end{align*}
Then we have that $ug=gv$, i.e. $(u,v^{-\tau})g=g$
and that $\chi(u,v^{-\tau})=\psi_{F}(y_{21}c_{12}+y_{22}c_{22}) \neq 1$
for suitably chosen $c_{12}, c_{22} \in F$, because $y_{21}$ and $y_{22}$ can not both be $0$. By Remark \ref{rem 2.3}, we deduce that there is no non-zero $(H,\chi)$-euqivariant distribution on $H\mathfrak{G}_{R}=G_R$.

\subsubsection{ Case:\ $R=\begin{pmatrix}
    3 & 1 \\ 1 & 1
\end{pmatrix}$}\ 
In this case, every matrix in the given $G_R$ is in the same $H$ orbit as elements in $\mathfrak{G}_R$, given by 
\begin{align*}
    g=\begin{pmatrix}
        x  & 0 & 0 \\ 0 & y & 0 \\ 0 & 0 & 1_2
    \end{pmatrix}x_0\begin{pmatrix}
        z  & 0 & 0 \\ 0 & w & 0 \\ 0 & 0 & 1_2
    \end{pmatrix}, \quad {\rm with}\quad x_0=\begin{pmatrix}
        E_{11}  & 0 & E_{22} \\ 0 & 1_2 & 0 \\ E_{22} & 0 & E_{11}
    \end{pmatrix},
\end{align*}
where $x,y,z,w \in \GL_2(F)$. 
Write
\begin{align*}
    \mathfrak{G}_R^{\prime}=\{ 
    g \in \mathfrak{G}_R \colon (y^{-1}x)_{12}=w_{12}, (y^{-1}x)_{22}=w_{22}, (zw^{-1})_{21}=y_{21}, \; \mathrm{and} \;(zw^{-1})_{22}=y_{22}
    \}.
\end{align*}
On the submanifold, $H(\mathfrak{G}_R \setminus \mathfrak{G}_R^{\prime})$, write
\small{\begin{align*}
    u=\begin{pmatrix}
        1_2 & x\begin{pmatrix}
            0 & 0\\
            b_{21} & b_{22}
        \end{pmatrix}y^{-1} & 0\\
        0 & 1_2 & y\begin{pmatrix}
            0 & c_{12}\\
            0 & c_{22}
        \end{pmatrix}\\
        0 & 0 & 1_2
    \end{pmatrix}
\quad {\rm and}\quad 
    v= \begin{pmatrix}
        1_2 & 0 & 0\\
        w^{-1}\begin{pmatrix}
            0 & c_{12}\\
            0 & c_{22}
        \end{pmatrix}z & 1_2 & 0\\
        0 & \begin{pmatrix}
            0 & 0\\
            b_{21} & b_{22}
        \end{pmatrix}w & 1_2
    \end{pmatrix}.
\end{align*}}
Then we have that $ug=gv,$ i.e. $(u,v^{-\tau})g=g.$
Since $g \notin \mathfrak{G}_R^{\prime}$,
\begin{align*}
    &\chi(u,v^{-\tau})\\
    &\quad =\psi((y^{-1}x)_{12}b_{21}+(y^{-1}x)_{22}b_{22}+y_{21}c_{12}+y_{22}c_{22}-(zw^{-1})_{21}c_{12}-(zw^{-1})_{22}c_{22}-w_{12}b_{21}-w_{22}b_{22}) \\
    &\quad \neq 1.
\end{align*}
for suitably chosen $b_{21}, b_{22}, c_{12}$ and $c_{22}$. 
By Remark \ref{rem .2}, we deduce that there is no non-zero $(H, \chi)$-equivariant distribution on $H(\mathfrak{G}_R \setminus \mathfrak{G}_R^{\prime})$.

On the submanifold $H \mathfrak{G}_R^{\prime}$, write
\small{\begin{align*}
    h_1=\begin{pmatrix}
        a & 0 & z^{\tau}\begin{pmatrix}
        0 & (x^{\tau}a^{\tau} z^{-1})_{12} \\ a_{12} & 0
    \end{pmatrix} \\ 0 & a & 0 \\ 0 & 0 & a
    \end{pmatrix}\quad  {\rm and}\quad  
    h_2=h_1^{\tau}=\begin{pmatrix}
        a^{\tau} & 0 & 0 \\ 0 & a^{\tau} & 0\\ \begin{pmatrix}
        0 & a_{12} \\ ((z^{\tau})^{-1}ax)_{21} & 0
    \end{pmatrix}z & 0 & a^{\tau}
    \end{pmatrix},
\end{align*}}
where $a=\begin{pmatrix}
        a_{11} & a_{12} \\ 0 & a_{22}
    \end{pmatrix}$, with $a_{11}(yw)_{12}+a_{12}(yw)_{22}=a_{22}(yw)_{21}$.
Then we have that $h_1g=g^{\tau}h_2$, i.e. $(h_1,h_2^{-\tau})g=g^{\tau}$. 
It is clear that $\chi(h_1,h_2^{-\tau})=1.$
By Lemma \ref{lem 2.7}, the $H$ orbit $Hg$ is $\tau$-invariant and any $(H,\chi)$-equivariant distribution on this $H$ orbit is $\tau$-invariant. In other words, there is no non-zero $(\widetilde{H},\widetilde{\chi})$-equivariant distribution on this orbit. By Remark \ref{rem 2.3}, we get that there is no non-zero $(\widetilde{H},\widetilde{\chi})$-equivariant distribution on $H \mathfrak{G}_R^{\prime}$.

\subsubsection{Case:\ $R=\begin{pmatrix}
        2 & 1 \\ 1 & 0 
    \end{pmatrix}$}\ 
In this case, every matrix in the given $G_R$ is in the same $H$ orbit as elements in $\mathfrak{G}_R$, given by 
\begin{align*}
    g=\begin{pmatrix}
        x  & 0 & 0 \\ 0 & y & 0 \\ 0 & 0 & 1_2
    \end{pmatrix}x_0\begin{pmatrix}
        z  & 0 & 0 \\ 0 & w & 0 \\ 0 & 0 & 1_2
    \end{pmatrix},\quad {\rm with}\quad x_0=\begin{pmatrix}
        0  & E_{12} & E_{22} \\ E_{21} & 0 & E_{11} \\ E_{22} & E_{11} & 0
    \end{pmatrix},
\end{align*}
where $x,y,z,w \in \GL_2(F)$. 
Write $\mathfrak{G}_R^{\prime}$ to be the subset of $\mathfrak{G}_R$ that consists of elements with the following properties:
\begin{align*}
      y_{11}=w_{11},\ (y^{-1}x)_{21}=(zw^{-1})_{12},\ (y^{-1}x)_{22}=0,\ (zw^{-1})_{22}=0,\ y_{21}=0,\ \mathrm{and} \ w_{12}=0.
\end{align*}
On the submanifold $H(\mathfrak{G}_R \setminus \mathfrak{G}_R^{\prime})$, write
\small{\begin{align*}
    u=\begin{pmatrix}
        1_2 & x\begin{pmatrix}
            0 & b_{12}\\
            0 & b_{22}
        \end{pmatrix}y^{-1} & x\begin{pmatrix}
            0 & d_{12} \\ d_{21} & 0
        \end{pmatrix}\\
        0 & 1_2 & y\begin{pmatrix}
            c_{11} & c_{12}\\
            0 & 0
        \end{pmatrix}\\
        0 & 0 & 1_2
    \end{pmatrix}\quad {\rm and}\quad 
    v= \begin{pmatrix}
        1_2 & 0 & 0\\
        w^{-1}\begin{pmatrix}
            0 & 0\\
            b_{12} & d_{12}
        \end{pmatrix}z & 1_2 & 0\\
          \begin{pmatrix}
            0 & c_{12}\\
            b_{22} & 0
        \end{pmatrix}z & \begin{pmatrix}
            c_{11} & 0\\
            d_{21} & 0
        \end{pmatrix}w & 1_2
    \end{pmatrix}.
\end{align*}}
Then have that $ug=gv$, i.e. $(u,v^{-\tau})g=g$.
Since $g \notin \mathfrak{G}_R^{\prime}$,
\begin{align*}
    \chi(u,v^{-\tau})&=\psi((y^{-1}x)_{21}b_{12}+(y^{-1}x)_{22}b_{22}+y_{11}c_{11}+y_{21}c_{12}\\
    &\qquad -(zw^{-1})_{12}b_{12}
    -(zw^{-1})_{22}d_{12}-w_{11}c_{11}-w_{12}d_{21}) \neq 1.
\end{align*}
for suitably chosen $b_{12}, b_{22}, c_{11},c_{12}, d_{12}, d_{21} \in F$. 
By Remark \ref{rem .2}, we get that there is no non-zero $(H, \chi)$-equivariant distribution on $H(\mathfrak{G}_R \setminus \mathfrak{G}_R^{\prime})$.

On the submanifold $H \mathfrak{G}_R^{\prime}$, write
\small{\begin{align*}
    h_1=\begin{pmatrix}
        a & 0 & z^{\tau}\begin{pmatrix}
        (y^{\tau}aw^{-1})_{21} & 0 \\ 0 & 0
    \end{pmatrix} \\ 0 & a & w^{\tau}\begin{pmatrix}
        0 & 0 \\
        0 & (x^{\tau}az^{-1})_{12}
    \end{pmatrix} \\ 0 & 0 & a
    \end{pmatrix},\quad 
    h_2=h_1^{\tau}=\begin{pmatrix}
        a & 0 & 0 \\ 0 & a & 0\\ \begin{pmatrix}
        (w^{-\tau}ay)_{12} & 0 \\ 0 & 0
    \end{pmatrix}z & \begin{pmatrix}
        0 & 0 \\
        0 & (z^{-\tau}ax)_{21}
    \end{pmatrix}w & a
    \end{pmatrix},
\end{align*}}
where $a=\begin{pmatrix}
        1 & 0 \\ 0 & \frac{x_{12}}{z_{21}}
    \end{pmatrix}$. 
Then we have that $h_1g=g^{\tau}h_2$, i.e. $(h_1,h_2^{-\tau})g=g^{\tau}$.
It is clear that $\chi(h_1,h_2^{-\tau})=1$.
By Lemma \ref{lem 2.7}, the $H$ orbit $Hg$ is $\tau$-invariant and any $(H,\chi)$-equivariant distribution on this $H$ orbit is $\tau$-invariant. In other words, there is no non-zero $(\widetilde{H},\widetilde{\chi})$-equivariant distribution on this orbit. By Remark \ref{rem 2.3}, we obtain that there is no $(\widetilde{H},\widetilde{\chi})$-equivariant distribution on $H \mathfrak{G}_R^{\prime}$.

\bibliographystyle{alpha}
	\bibliography{models}

\end{document}